 \documentclass{amsart}
\usepackage{graphicx}

\usepackage{amsfonts}
\usepackage{amssymb}
\usepackage{amsmath}
\usepackage{amsthm}
\usepackage{amscd}

\def\vbar{\mathchoice{\vrule height6.3ptdepth-.5ptwidth.8pt\kern- .8pt}
{\vrule height6.3ptdepth-.5ptwidth.8pt\kern-.8pt} {\vrule
height4.1ptdepth-.35ptwidth.6pt\kern-.6pt} {\vrule
height3.1ptdepth-.25ptwidth.5pt\kern-.5pt}}
\def\fudge{\mathchoice{}{}{\mkern.5mu}{\mkern.8mu}}
\def\bbc#1#2{{\rm \mkern#2mu\vbar\mkern-#2mu#1}}
\def\bbb#1{{\rm I\mkern-3.5mu #1}}
\def\bba#1#2{{\rm #1\mkern-#2mu\fudge #1}}
\def\bb#1{{\count4=`#1 \advance\count4by-64 \ifcase\count4\or\bba
A{11.5}\or \bbb B\or\bbc C{5}\or\bbb D\or\bbb E\or\bbb F \or\bbc
G{5}\or\bbb H\or \bbb I\or\bbc J{3}\or\bbb K\or\bbb L \or\bbb
M\or\bbb N\or\bbc O{5} \or \bbb P\or\bbc C{5}\or\bbb B\or\bbc
S{4.2}\or\bba T{10.5}\or\bbc U{5}\or \bba V{12}\or\bba
W{16.5}\or\bba X{11}\or\bba Y{11.7}\or\bba Z{7.5}\fi}}

\newcommand{\K}{{\mathbb{K}}}

\newtheorem{df}{Definition}[section]
\newtheorem{thm}{Theorem}[section]
\newtheorem{cor}{Corollary}[section]
\newtheorem{rem}{Remark}[section]

\newtheorem{prop}{Proposition}[section]

\newtheorem{lem}{Lemma}[section]

\begin{document}

\date{}
\title{BiHom pre-Lie superalgebras and related structures}
\author{OTHMEN NCIB }
\address{University of Gafsa, Faculty of Sciences Gafsa, 2112 Gafsa, Tunisia}
\email{othmenncib@yahoo.fr}

 \maketitle{}

 \begin{abstract}
 Throughout this paper, we will study Rota-Baxter operators and super $\mathcal{O}$-operator of associative BiHom superalgebras, BiHom Lie superalgebras, BiHom pre-Lie superalgebras and BiHom-$L$-dendriform superalgebras. Then we give some properties of BiHom pre-Lie superalgebras constructed from associative BiHom superalgebras, BiHom Lie superalgebras and BiHom-$L$-dendriform superalgebras. As applications, we study the theory of cohomology groups of BiHom-pre-Lie superalgebras.
 \end{abstract}
\section*{Introduction}

The theory of Hom-structures may found in the physics literature around $1990$, conserning $q$-deformation of algebras of vector fields, especially Witt and Virasoro algebras (see \cite{Aizawa-N-Sato-H, Chaichian-Kulish-Lukierski, Curtright-Zachos, Liu-K-Q} ).

A BiHom-algebra is an algebra in such a way that the identities defining the structure are twisted by two homomorphisms $\alpha,\;\beta$. This class of algebras was introduced from a categorical approach in \cite{Graziani-Makhlouf-Menini-Panaite} as an extension of the class of Hom-algebras. More applications of BiHom-Lie algebras, biHom-algebras, BiHom-Lie superalgebras and BiHom-Lie admissible superalgebras can be found in (\cite{Cheng-Qi, S-Wang-S-Guo}).\\

A Rota-Baxter algebra is an associative algebras together with a linear operator that satisfies an identity abstructed from the integration by part formula in calculus. The study of Rota-Baxter algebra appeared for the first time in the work of the mathematician G. Baxter \cite{Baxter1960} in $1960$ and were then intensively studied by F. V. Atkinson \cite{Atkinson1967}, J. B. Miller \cite{Miller1969}, G.-C. Rota \cite{Rota1995}, P. Cartier \cite{Cartier1972} and  more recently they reappeared in the work of L. Guo \cite{Guo-Introd} and K. Ebrahimi-Fard \cite{Ebrahimi-loday.algebras}.\\

Pre-Lie algebras are a class of nonassociative algebras coming from the study of
convex homogeneous cones \cite{Vinberg}, affine manifolds and affine structures on Lie groups \cite{Koszul}, and the aforementioned cohomologies of associative algebras \cite{Gerstenhaber}. They also appeared in many fields in mathematics and mathematical physics, such as complex and symplectic structures on Lie groups and Lie algebras \cite{Andrada-Salamon,Chu,Dardié-Medina1,Dardié-Medina2,Lichnerowicz-Medina}, phases spaces of Lie algebras \cite{Bai phase spaces,Kupershmidt2}, integrable systems \cite{Bordemann}, classical and quantum Yang-Baxter equations \cite{Diata-Medina}, combinatorics \cite{Ebrahimi-loday.algebras}, Poisson brackets and infinite dimensional Lie algebras, vertex algebras, quantum field theory \cite{Connes-Kreimer}, and operads \cite{Chapoton-Livernet}. See \cite{Burde} for  a survey. Recently, pre-Lie superalgebras, the  $\mathbb{Z}_2$-graded version of pre-Lie algebras, also appeared in many others fields; see for example \cite{Chapoton-Livernet,Gerstenhaber,Mikhalev}. To our knowledge, they were first introduced by Gerstenhaber in $1963$ to study the cohomology structure of associative algebras \cite{Gerstenhaber}. They are a class of natural algebraic appearing in many fields in mathematics and mathematical physics, especially in  super-symplectic geometry, vertex superalgebras and graded classical Yang-Baxter equation \cite{Bai O-operators,Bai-algebraic}. Recently, the classifications of complex pre-Lie superalgebras in dimensions two and three were given by R. Zhang and C.M.  Bai \cite{Bai and Zhang classif}\\

 It is clear that the  Rota-Baxter operator is to use in several constructions, one of the most important is the construction of BiHom-pre-Lie superalgebra from a BiHom associative superalgebra. Let $(\mathcal{A},\cdot,\alpha,\beta)$ be a BiHom associative superalgebra and $R$ be a Rota-Baxter operator of weight $\lambda$ on $\mathcal{A}$, which means that it satisfies, for any homogeneous elements $x,y$ in $\mathcal{A}$, the identity
\begin{equation}\label{R-B assoc}
R(x)\cdot R(y)=R\Big( R(x)\cdot y+xR(y)+\lambda x\cdot y\Big).
\end{equation}
If $\lambda=0$ (resp. $\lambda=-1$),  the product
\begin{equation}\label{intr-ass=pre-Lie0}
    x\circ y=R(x)\cdot y-(-1)^{|x||y|}\alpha^{-1}\beta(y) \cdot R(\alpha\beta^{-1}(x)),~~\forall~~x,y\in\mathcal{H}(\mathcal{A})
\end{equation}
resp.
\begin{equation}\label{intr-ass=pre-Lie}
    x\circ y=R(x)\cdot y-(-1)^{|x||y|}\alpha^{-1}\beta(y) \cdot R(\alpha\beta^{-1}x)-x\cdot y,~~\forall~~x,y\in\mathcal{H}(\mathcal{A})
\end{equation}
defines a BiHom-pre-Lie superalgebra $($see Theorem \ref{RB+ass=BiHompreLie super}$)$.\\

The notion of dendriform algebra was introduced by Loday (\cite{Loday-dialgebras}) in $1995$ with motivation from algebraic $K$-theory and has been studied quite extensively with connections to several areas in mathematics and physics, including operads, homology, Hopf algebras, Lie and Leibniz algebras, combinatorics, arithmetic and quantum field theory and so on (see \cite{Ebrahimi-Manchon-Patr} and the references therein). The relationship between dendriform algebras, Rota-Baxter algebras and pre-Lie algebras was given by M. Aguiar and K. Ebrahimi-Fard \cite{Aguiar2000,Ebrahimi-loday.algebras,Ebrahimi-assoc-Nijenhuis}. C. Bai, L. Liu, L. Guo and X. Ni, generalized the concept of Rota-Baxter operator and introduced a new class of algebras, namely, $L$-dendriform algebras, in
\cite{ Bai-Liu OOperaorLie,Bai-Liu OOperaorAss,Bai-Liu L-dendrifom}. Moreover, there is the following relationship among Lie superalgebras, associative superalgebras, pre-Lie superalgebras and dendriform superalgebras in the sense of commutative diagram of categories:
\begin{equation}\label{diag1}
\begin{array}{ccc}
  \text{Lie superalgebra} & \longleftarrow & \text{pre-Lie superalgebra}\\
  \uparrow &  & \uparrow \\
  \text{associative superalgebra} & \longleftarrow & \text{dendriform superalgebra}
\end{array}
\end{equation}

Later quite a few more similar algebra structures have been introduced, such as quadrialgebras of Aguiar and Loday \cite{Aguiar-Loday2004}. In order to extend the commutative diagram \eqref{diag1} at the level of associative superalgebras (the
bottom level of the commutative diagram \eqref{diag1}) to the more Loday superalgebras, it is natural to
find the corresponding algebraic structures at the level of Lie superalgebras which extends the top
level of commutative diagram \eqref{diag1}. We will show that the $L$-dendriform superalgebras are chosen in
a certain sense such that the following diagram including the diagram \eqref{diag1} as a sub-diagram is
commutative:
\begin{equation}\label{diag2}
\begin{array}{ccc}
  \text{Lie superalgebra} & \longleftarrow  \text{pre-Lie superalgebra} &   \longleftarrow  \text{$L$-dendriform superalgebra}\\
  \uparrow &   \uparrow &  \uparrow  \\
  \text{associative superalgebra} & \longleftarrow  \text{dendriform superalgebra}& \longleftarrow  \text{ quadri-superalgebra}
\end{array}
\end{equation}

Recently, the notion of Rota-Baxter operator on a bimodules was introduced by M. Aguiar \cite{Aguiar2004}. The construction of associative, Lie, pre-Lie and $L$-dendriform superalgebras are extended to the corresponding categories of bimodules.\\
The main purpose of this paper is to study, through Rota-Baxter operators and $\mathcal{O}$-operators,  the relationship between BiHom associative superalgebras, BiHom-Lie superalgebras, BiHom-pre-Lie superalgebras and BiHom-$L$-dendriform superalgebras.\\
 This paper is organized as follows. In Section $1$, we recall some definitions of BiHom associative superalgebras, BiHom-Lie superalgebras and BiHom-pre-Lie superalgebras and we introduce the notion of super $\mathcal{O}$-operator of these BiHom superalgebras that generalizes the notion of Rota-Baxter operators. We show that every Rota-Baxter BiHom associative superalgebra of weight $\lambda=-1$ gives rise to a Rota-Baxter BiHom-Lie superalgebra. Moreover, a super $\mathcal{O}$-operator on  a BiHom-Lie superalgebra $($of weight zero$)$ gives rise to  a BiHom-pre-Lie superalgebra. In Section $2$, we introduce the notion of BiHom-$L$-dendriform superalgebra and then study some fundamental properties of BiHom-$L$-dendriform superalgebras in terms of  super $\mathcal{O}$-operator of BiHom-pre-Lie superalgebras. Their relationship with BiHom associative superalgebras are also described. In Section $3$, we give the cohomology theory of a BiHom-pre-Lie superalgebra $(\mathcal{A},\cdot,\alpha,\beta)$ with the coefficient in a representation $(V,l,r,\alpha_V,\beta_V)$. The main contribution is to
define the coboundary operator $ \delta: C^n(\mathcal{A}; V )\longrightarrow C^{n+1}(\mathcal{A}; V )$.\\
Throughout this paper, all superalgebras are finite-dimensional and are over a field $\mathbb{K}$ of characteristic zero.  Let $(\mathcal{A},\circ,\alpha,\beta)$ be a BiHom superalgebra, then $L_\circ$ and $R_\circ$ denote the even  left and right multiplication operators $L_\circ,R_\circ:\mathcal{A}\rightarrow End(\mathcal{A})$ defined as  $L_\circ(x)(y)=(-1)^{|x||y|}R_\circ(y)(x)=x\circ y$ for all homogeneous element $x,y$ in $\mathcal{A}$.
 In particular, when $(\mathcal{A},[ \ ,\ ],\alpha,\beta)$ is a BiHom-Lie superalgebra, we let $ad(x)$ denote the adjoint operator, that is, $ad(x)(y)=[x,y]$ for all homogeneous element $x,y$ in $\mathcal{A}$.

 \section{ Rota-Baxter BiHom associative superalgebras, BiHom pre-Lie superalgebras and BiHom Lie superalgebras}
Let $(\mathcal{A},\circ)$ be an algebra over a field $\mathbb{K}$. It is said to be a superalgebra if the underlying vector space of $\mathcal{A}$ is $\mathbb{Z}_2$-graded, that is, $\mathcal{A}=\mathcal{A}_0\oplus \mathcal{A}_1$, and $\mathcal{A}_i\circ \mathcal{A}_j\subset \mathcal{A}_{i+j}$, for $i,j\in \mathbb{Z}_2$. An element of $\mathcal{A}_0$ is said to be  even and an element of $\mathcal{A}_1$ is said to be odd. The elements of $\mathcal{A}_{j},~~j \in \mathbb{Z}_2$, are said to be homogenous and of parity $j$. The parity of a homogeneous element $x$ is denoted by $|x|$ and we refer to the set of homogeneous elements of $\mathcal{A}$ by $\mathcal{H}(\mathcal{A})$.\\

We extend to graded case the concepts of  $\mathcal{A}$-bimodule $\K$-algebra,  $\mathcal{O}$-operator and extended  $\mathcal{O}$-operator introduced in \cite{Bai-Liu OOperaorAss}.
\begin{df}
An associative superalgebra is a pair $(\mathcal{A},\mu)$ cosisting of a $\mathbb{Z}_2$-graded vector space $\mathcal{A}$ and an even bilinear map $\mu:\mathcal{A}\otimes\mathcal{A}\longrightarrow\mathcal{A}$, (i.e\;:\;$\mu(\mathcal{A}_i,\mathcal{A}_j)\subseteq\mathcal{A}_{i+j},\forall i,j\in\mathbb{Z}_2$) satisfying for any $x,y,z \in \mathcal{A}$
$$\mu(x,\mu(y,z))=\mu(\mu(x,y),z).$$
\end{df}

\begin{df}\ \begin{enumerate}
\item A BiHom-associative superalgebra is a tuple  $(\mathcal{A},\mu,\alpha,\beta)$ consisting of a linear space $\mathcal{A}$, a bilinear map $\mu:\mathcal{A}\otimes\mathcal{A}\longrightarrow \mathcal{A}$ and two even linear maps $\alpha,\beta:\mathcal{A}\longrightarrow \mathcal{A} $, satisfying the following conditions for all $a,b,c \in \mathcal{A}$.
\begin{eqnarray*}
 \alpha \circ \beta &=& \beta \circ \alpha, \\
 \alpha(\mu(a,b))&=&\mu(\alpha(a),\alpha(b))\;\; and\;\;\beta(\mu(a,b))=\mu(\beta(a),\beta(b))\;\;(multiplicativity),\\
 \mu(\alpha(a),\mu(b,c))&=&\mu(\mu(a,b),\beta(c)) (BiHom-associativity).
\end{eqnarray*}
A morphism $f : (\mathcal{A}, \mu_\mathcal{A}, \alpha_\mathcal{A}, \beta_\mathcal{A}) \rightarrow (\mathcal{B}, \mu_\mathcal{B}, \alpha_\mathcal{B}, \beta_\mathcal{B})$ of BiHom-associative superalgebras is a linear
map $f : \mathcal{A} \rightarrow \mathcal{B}$ such that $\alpha_\mathcal{B} \circ f = f \circ \alpha_\mathcal{A},\; \beta_\mathcal{B} \circ f = f \circ \beta_\mathcal{A}$
 and $f \circ \mu_\mathcal{A} = \mu_\mathcal{B} \circ (f \otimes f)$.

\item
 Let $(\mathcal{A},\mu,\alpha,\beta)$ be a BiHom associative superalgebras and $V$ be a $\mathbb{Z}_2$-graded vector space. Let $l,\;r\;:\;\mathcal{A}\rightarrow End(V)$ be two even linear maps and $\alpha_V,\;\beta_V\;:\;V\rightarrow V$ be two even linear maps. A tuple $(V,l,r,\alpha_V,\beta_V)$ is called an $\mathcal{A}$-bimodule if for all $x,\;y\in\mathcal{H}(\mathcal{A})$ and $v\in\mathcal{H}(V)$
\begin{eqnarray*}
l(\alpha(x))l(y)(v)&=&l(\mu(x,y))(\beta_V(v))\\
l(\alpha(x))r(y)(v)&=&r(\beta(y))l(x)(v)\\
r(\mu(x,y))\alpha_V(v)&=&r(\beta(y))r(x)(v)\\
l(\alpha(x))\alpha_V(v)&=&\alpha_V(l(x))(v)\\
r(\alpha(x))\alpha_V(v)&=&\alpha_V(r(x)(v))\\
l(\beta(x))(\beta_V(v))&=&\beta_V(l(x)(v))\\
r(\beta(x))(\beta_V(v))&=&\beta_V(r(x)(v))\\
\end{eqnarray*}
Moreover, the tuple $(V,\mu_V,l,r,\alpha_V,\beta_V)$ is said to be an  $\mathcal{A}$-bimodule $\mathbb{K}$-superalgebra if $(V,l,r,\alpha_V,\beta_V)$  is an  $\mathcal{A}$-bimodule compatible with the multiplication $\mu_V$ on $V$, that is, for all $x,y \in\mathcal{H}(\mathcal{A})$ and $u,v\in \mathcal{H}(V),$
\begin{eqnarray*}
l(\alpha(x))\mu_V(u,v)&=&\mu_V(l(x)(u),\beta_V(v)),\\
\mu_V(\alpha_V(u),r(x)(v))&=&r(\beta(x))(\mu_V(u,v)),\\
\mu_V(\alpha_V(u),l(x)(v))&=&\mu_V(r(x)(u),\beta_V(v)).
\end{eqnarray*}

 \item Fix $\lambda\in \K$.  A pair $(T,T')$ of even linear maps $T,T':V\longrightarrow \mathcal{A}$ is called an extended super $\mathcal{O}$-operator with modification $T'$ of
weight $\lambda$  associated to the bimodule $(V,l,r,\alpha_V,\beta_V)$ if $T$ satisfies
\begin{eqnarray}
& \alpha\circ T=T\circ\alpha_V\;\;and\;\;\beta\circ T=T\circ\beta_V,\\
& \alpha\circ T'=T'\circ\alpha_V\;\;and\;\;\beta\circ T'=T'\circ\beta_V,\\
& \lambda l(T'(u))v=\lambda r(T'(v))u,\\
&  T(u)T(v)=T\Big(l(T(u))v+(-1)^{|u||v|}r(T(v))u\Big)+\lambda T'(u)T'(v),~~\forall~~u,v \in \mathcal{H}(V).
 \end{eqnarray}
 \item  An even  linear map $T:V\longrightarrow \mathcal{A}$ is called a super $\mathcal{O}$-operator of
weight $\lambda$  associated to the bimodule $\K$-superalgebra  $(V,\mu_V,l,r,\alpha_V,\beta_V)$ if it satisfies
\begin{eqnarray}
& \alpha\circ T=T\circ\alpha_V\;\;and\;\;\beta\circ T=T\circ\beta_V,\nonumber\\
&  T(u)T(v)=T\Big(l(T(u))v+(-1)^{|u||v|}r(T(v))u+\lambda \mu_V(u,v)\Big),~~\forall~~u,v \in \mathcal{H}(V).
 \end{eqnarray}
\end{enumerate}
\end{df}
Notice that the notions of  super $\mathcal{O}$-operator and extended  super $\mathcal{O}$-operator coincide when $\lambda =0$.\\



In particular, a super $\mathcal{O}$-operator of weight $\lambda\in \mathbb{K} $ associated to the bimodule $\K$-superalgebra $(\mathcal{A},\mu_A,L_\mu,R_\mu,\alpha,\beta)$ is called a Rota-Baxter operator of weight $\lambda$ on $\mathcal{A}$, that is, $R$ satisfies  the identity \eqref{R-B assoc}.
We denote by a tuple $(\mathcal{A},\mu,R,\alpha,\beta)$ the Rota-Baxter BiHom associative superalgebra.\\


We define now Rota-Baxter operators on $\mathcal{A}$-bimodules.
\begin{df}Let $(\mathcal{A},\mu,R,\alpha,\beta)$ be a Rota-Baxter BiHom associative superalgebra of weight zero. A Rota-Baxter operator on an  $\mathcal{A}$-bimodule $V$ $($relative to $R)$ is a map $R_V:V\longrightarrow V$ such that for all
 $x \in \mathcal{H}(\mathcal{A})$ and $v \in \mathcal{H}(V)$
\begin{eqnarray*}
    & & \ \ \ \alpha_V\circ R_V=R_V\circ\alpha_V \;\; and\;\;\beta_V\circ R_V=R_V\circ\beta_V,\\
    & & \ \ \ R(x)R_V(v) = R_V \Big(R(x)v+x R_V(v)\Big),\\
    & & \ \ \ R_V(v)R(x) = R_V \Big(R_V(v)x+v R(x) \Big).
    \end{eqnarray*}
\end{df}
We have similar definitions on BiHom Lie superalgebras.
 \begin{df}\ \begin{enumerate}
\item A BiHom Lie superalgebra is a tuple $(\mathcal{A}, [~~,~~],\alpha,\beta)$ consisting of a $\mathbb{Z}_2$-graded vector space $\mathcal{A}$, an even
bilinear map $[~~,~~] : \mathcal{A}\otimes \mathcal{A} \longrightarrow \mathcal{A},~~([\mathcal{A}_i,\mathcal{A}_j]\subseteq \mathcal{A}_{i+j},~~\forall~~i,j\in \mathbb{Z}_2)$ and two even linear maps $\alpha,\beta:\mathcal{A}\rightarrow\mathcal{A}$  satisfying for all $\ x,y,z \in \mathcal{H}(\mathcal{A})$,
\begin{eqnarray}
 &&\alpha\circ\beta =\beta\circ\alpha,\\
 \label{B-H-skwesym}
 &&[\beta(x),\alpha(y)] = -(-1)^{|x||y|}[\beta(y),\alpha(x)],\quad \text{(BiHom super skew-symmetry)}\\
\label{B-H-sJ}
&& \displaystyle\circlearrowleft_{x,y,z}(-1)^{|x||z|} [\beta^2(x),[\beta(y),\alpha(z)]]=0 , \quad \text{(BiHom super-Jacobi identity)}.
\end{eqnarray}
\item
 A representation of a BiHom-Lie superalgebra $(\mathcal{A},[\cdot,\cdot],\alpha,\beta)$ on
 a vector superspace $V$ with respect to commuting even linear maps $\alpha_V,\beta_V:V\rightarrow V$ is an even linear map
  $\rho:\mathcal{A}\longrightarrow End(V)$, such that for all
  $x,y\in \mathcal{H}(\mathcal{A})$, the following equalities are satisfied
\begin{eqnarray}
\label{bihom-lie-rep-1}\rho(\alpha(x))\circ \alpha_V&=&\alpha_V\circ \rho(x),\\
\label{bihom-lie-rep-2} \rho(\beta(x))\circ \beta_V&=&\beta_V \circ \rho(x),\\
\label{bihom-lie-rep-3}\rho([\beta(x),y])\circ  \beta_V &=&\rho(\alpha \beta(x))\circ\rho(y)-(-1)^{|x||y|}\rho(\beta(y))\circ\rho(\alpha(x)).
\end{eqnarray}
The pair $(V,\rho,\alpha_V,\beta_V)$ is said to be an $\mathcal{A}$-module or a representatin of $(\mathcal{A},[\cdot,\cdot],\alpha,\beta)$.
 \\ The tuple $(V,[ \ ,\ ]_V,\rho,\alpha,\beta)$, where $[ \ ,\ ]_V$ is a super skew-symmetric bracket,  is said to be an $\mathcal{A}$-module $\K$-superalgebra if,  for $x\in \mathcal{H}(\mathcal{A})$ and $v,w\in \mathcal{H}(V)$
 $$ \rho(\beta^2(x))[ \beta_V(v),\alpha_V(w) ]_V=[\rho(\beta(x))(\alpha_V(v)), \beta_V^2(w)]_V+(-1)^{|x||v|}[ \beta_V^2(v),\rho(\beta(x))(\alpha_V(w))]_V.$$
\item Let $(\mathcal{A},[\cdot,\cdot],\alpha,\beta)$ be a BiHom-Lie superalgebra and $(V,\rho,\alpha_V,\beta_V)$ be a representation.  An even linear map $T: V \to A$ is called a super $\mathcal{O}$-\textbf{operator} of weight $\lambda\in\mathbb{K}$
associated with an $\mathcal{A}$-module $\mathbb{K}$-superalgebra $(V,[ \ ,\ ]_V,\rho,\alpha,\beta)$, if it satisfies
\begin{eqnarray}
& \alpha\circ T=T\circ\alpha_V\;\;and\;\;\beta\circ T=T\circ\beta_V,\\
\label{O-operator} &[T(u),T(v)]=T\Big(\rho(T(u))v-(-1)^{|u||v|}\rho(T(\alpha_V^{-1}\beta_V(v)))\alpha_V\beta_V^{-1}(u)+\lambda[u,v]_V\Big),\ \ \forall u,v \in \mathcal{H}(V).
\end{eqnarray}
\end{enumerate}
 In particular, a super $\mathcal{O}$-operator of weight $\lambda \in \mathbb{K}$ associated to the bimodule $(\mathcal{A},L_\circ,R_\circ,\alpha,\beta)$ is called a Rota-Baxter operator of weight $\lambda \in \mathbb{K}$ on $(\mathcal{A},[~~,~~],\alpha,\beta)$, that is, $R$ satisfies for all $x,y,z$ in $\mathcal{H}(\mathcal{A})$
\begin{equation}\label{Rota-baxter-Lie}
    [R(x),R(y)]=R\Big([R(x),y]+[x,R(y)]+\lambda [x,y]\Big).
\end{equation}

The tuple $(\mathcal{A},[~~,~~],R,\alpha,\beta)$ refers to a Rota-Baxter BiHom-Lie superalgebra .
\end{df}
\begin{df}
Let $(\mathcal{A},[ ~~,~~ ],R,\alpha,\beta)$ and $(\mathcal{A}',[ ~~,~~ ]',R',\alpha',\beta')$ be two Rota-Baxter BiHom-Lie superalgebras. An even homomorphism $f:(\mathcal{A},[ ~~,~~ ],R,\alpha,\beta)\longrightarrow (\mathcal{A}',[ ~~,~~ ]',R',\alpha',\beta')$ is said to be a morphism of two Rota-Baxter BiHom-Lie  superalgebras if, for all $x,y\in \mathcal{H}(\mathcal{A})$,
\begin{eqnarray*}
&&f\circ\alpha=\alpha'\circ f\;\;\; and\;\;\; f\circ\beta=\beta'\circ f,\\
&&f\circ R = R'\circ f,\\
&&f([x,y]) = [f(x),f(y)]'.
\end{eqnarray*}
\end{df}
\begin{prop}
 Let $(\mathcal{A},\mu,R,\alpha,\beta)$ be a Rota-Baxter BiHom associative superalgebra of weight $\lambda \in \mathbb{K}$. Then the tuple  $(\mathcal{A},[ ~~,~~ ],R,\alpha,\beta)$, where $[x,y]=\mu(x,y)-(-1)^{|x||y|}\mu(y,x)$, is a Rota-Baxter BiHom-Lie superalgebra of weight $\lambda \in \mathbb{K}$.
\end{prop}
We introduce the notion of super $\mathcal{O}$-operators of BiHom-pre-Lie superalgebras and study some properties over BiHom-Lie superalgebras and BiHom-pre-Lie superalgebras.

\begin{df}\cite{Abdaoui-Mabrouk-Makhlouf}
Let $\mathcal{A}$ be a $\mathbb{Z}_2$ graded vector space and $\circ\;:\;\mathcal{A}\otimes\mathcal{A}\longrightarrow\mathcal{A}$ be an even binary
operation. The pair $(\mathcal{A},\circ)$ is called a pre-Lie superalgebra if, for $x,y,z$ in $\mathcal{H}(\mathcal{A})$, the associator
$$as(x,y,z)=(x\circ y)\circ z-x\circ( y\circ z)$$
is super-symmetric in $x$ and $y$, that is, $as(x,y,z)=(-1)^{|x||y|}as(y,x,z)$, or equivalently
 \begin{equation}\label{identitypreliesuper}
(x\circ y)\circ z-x\circ( y\circ z)=(-1)^{|x||y|}\Big((y\circ x)\circ z-y\circ( x\circ z)\Big)
\end{equation}
The identity (\ref{identitypreliesuper}) is called pre-Lie super-identity.
\end{df}
\begin{df}A BiHom-pre-Lie superalgebra is a tuple
$(\mathcal{A},\circ,\alpha,\beta)$ consisting of a $\mathbb{Z}_2$ graded vector space $\mathcal{A}$, an even binary
operation $\circ: \mathcal{A}\otimes \mathcal{A} \to \mathcal{A}$ and two even linear maps $\alpha,\beta: A \to A$ such that for any $x,y,z \in \mathcal{H}(\mathcal{A})$
\begin{align}
 \alpha \beta= &\beta  \alpha,\\
 \alpha(x \circ y)=\alpha(x)\circ \alpha(y)\;\;and\;\;  & \beta(x \circ y)=\beta(x)\circ \beta(y),\\
\label{id-Bihom-super-pre}(\beta(x)\circ \alpha(y))\circ \beta(z)-\alpha\beta(x)\circ (\alpha(y)\circ z)=(-1)^{|x||y|}&\Big(\big(\beta(y)\circ \alpha(x)\big)\circ \beta(z)-\alpha\beta(y)\circ \big(\alpha(x)\circ z\big)\Big).
\end{align}
The identity (\ref{id-Bihom-super-pre}) is called BiHom-pre-Lie super-identity.
\end{df}
\begin{prop}
Let $(\mathcal{A},\circ)$ be a pre-Lie superalgebra and $\alpha,\;\beta:\mathcal{A}\longrightarrow\mathcal{A}$ two commuting even morphisms. Then
$(\mathcal{A},\circ_{\alpha,\beta},\alpha,\beta)$ be a BiHom-pre-Lie superalgebra where

$$x\circ_{\alpha,\beta}y=\alpha(x)\circ\beta(y).$$

\end{prop}

\begin{proof}
Let $x,y,z\in\mathcal{H}(A)$, we have
\begin{eqnarray*}
as_{\alpha,\beta}(x,y,z)&=&\big(\beta(x)\circ_{\alpha,\beta}\alpha(y)\big)\circ_{\alpha,\beta}\beta(z)-
\alpha\beta(x)\circ_{\alpha,\beta}\big(\alpha(y)\circ_{\alpha,\beta}z\big)\\
&=&\alpha\big(\beta(x)\circ_{\alpha,\beta}\alpha(y)\big)\circ\beta^2(z)-\alpha^2\beta(x)\circ\beta\big(\alpha(y)\circ_{\alpha,\beta}z\big)\\
&=&\big(\alpha^2\beta(x)\circ\alpha^2\beta(y)\big)\circ\beta^2(z)-\alpha^2\beta(x)\circ\big(\alpha^2\beta(y)\circ\beta^2(z)\big).
\end{eqnarray*}
suppose that $X=\alpha^2\beta(x),\;Y=\alpha^2\beta(y)$ and $Z=\beta^2(z)$, then we get
\begin{eqnarray*}
as_{\alpha,\beta}(x,y,z)&=&(X\circ Y)\circ Z- X\circ(Y\circ Z)\\
&=&(-1)^{|X||Y|}\Big((Y\circ X)\circ Z- Y\circ(X\circ Z)\Big)\\
&=&(-1)^{|x||y|}\Big(\big(\alpha^2\beta(y)\circ \alpha^2\beta(x)\big)\circ\beta^2(z)-\alpha^2\beta(y)\circ\big(\alpha^2\beta(x)\circ \beta^2(z)\big)\Big)\\
&=&(-1)^{|x||y|}\Big(\big(\beta(y)\circ_{\alpha,\beta} \alpha(x)\big)\circ_{\alpha,\beta}\beta(z)
-\alpha\beta(y)\circ_{\alpha,\beta}\big(\alpha(x)\circ_{\alpha,\beta} z\big)\Big).
\end{eqnarray*}
We conclude that $\circ_{\alpha,\beta}$ define a BiHom-pre-Lie superalgebra structure on $\mathcal{A}$.
\end{proof}
\begin{df}\label{defi-oper-pre}\
 Let $V$ be a $\mathbb{Z}_2$-graded vector space, $l,r:\mathcal{A}\longrightarrow End(V)$ be two even linear maps and
 $\alpha_V,\beta_V:V\longrightarrow V$ be two even linear maps. The tuple $(V,l,r,\alpha_V,\beta_V)$ is said to be an $\mathcal{A}$-bimodule of $(\mathcal{A},\circ,\alpha,\beta)$
  if,  for   $x,y \in \mathcal{H}(\mathcal{A})$ and $u\in \mathcal{H}(V),$
 \begin{align}
&\alpha_Vl(x)(u)=l(\alpha(x)\alpha_V(u),~~  \beta_Vl(x)(u)=l(\beta(x))\beta_V(u),\label{rep1}\\ &\alpha_Vr(x)(u)=r(\alpha(x))\alpha_V(u),~~   r(\beta(x))\beta_V(u)=\beta_Vr(y)(u),\nonumber\\
&\Big(l(\beta(x)\circ\alpha(y))\beta_V-l(\alpha\beta(x))l(\alpha(y))\Big)(u)=
(-1)^{|x||y|}\Big(l(\beta(y)\circ\alpha(x))\beta_V-l(\alpha\beta(y))l(\alpha(x))\Big)(u),\label{rep2}\\
&\Big(r(\beta(x))r(\alpha(y))\beta_V-r(\alpha(y)\circ x)\alpha_V\beta_V\Big)(u)=
(-1)^{|y||u|}\Big(r(\beta(x))l(\beta(y)\alpha_V-l(\alpha\beta(y))r(x)\alpha_V\Big)(u),\label{rep3}
\end{align}
 \end{df}
 Moreover, the tuple $(V,\circ_V,l,r,\alpha_V,\beta_V)$ is said to be an  $\mathcal{A}$-bimodule $\K$-superalgebra if $(V,l,r,\alpha_V,\beta_V)$  is an  $\mathcal{A}$-bimodule compatible with the multiplication $\circ_V$ on $V$, that is, for $x \in \mathcal{H}(\mathcal{A})$ and $u,v\in \mathcal{H}(V),$
\begin{eqnarray*} && l(\beta(x))(\alpha_V(u))\circ_V(\beta_V( v))-l(\alpha\beta(x))(\alpha_V(u)\circ_Vv)\\
&&= (-1)^{|x||u|} \Big(r(\alpha(x))(\beta_V(u))\circ_V\beta_V(v)-\alpha_V\beta_V(u)\circ_Vl(\alpha(x))(v)\Big),\\
&& ~~ r(\beta(x))(\beta_V(u)\circ_V\alpha_V(v))-\alpha_V\beta_V(u)\circ_V(r(z)(\alpha_V(v)))\\
&&=
(-1)^{|u||v|}\Big(r(\beta(x))(\beta_V(v)\circ_V\alpha_V(u))-\alpha_V\beta_V(v)\circ_V(r(x)(\alpha_V(u)))\Big).
\end{eqnarray*}
 Let  $(V,\circ_V,l,r,\alpha_V,\beta_V)$ be an $\mathcal{A}$-bimodule $\K$-superalgebra. An even linear map
$T:V\longrightarrow \mathcal{A}$ is called a super $\mathcal{O}$-operator of weight $\lambda \in \mathbb{K}$ associated to $(V,\circ_V,l,r,\alpha_V,\beta_V)$ if it  satisfies:
\begin{eqnarray}
& \alpha\circ T=T\circ\alpha_V\;\;and\;\;\beta\circ T=T\circ\beta_V,\\
\label{opera2} & T(u)\circ T(v)= T\Big(l(T(u))v+(-1)^{|u||v|}r(T(\alpha_V^{-1}\beta_V(v))\alpha_V\beta_V^{-1}(u)) +\lambda u\circ_V v\Big),~~\forall~~~u,v\in \mathcal{H}(V).
\end{eqnarray}
 In particular, a super $\mathcal{O}$-operator of weight $\lambda \in \mathbb{K}$ associated to the  $\mathcal{A}$-bimodule $(\mathcal{A},L_\circ,R_\circ,\alpha,\beta)$ is called a Rota-Baxter operator of weight $\lambda$ on $(\mathcal{A},\circ,\alpha,\beta)$, that is, $R$ satisfies
\begin{equation}\label{Rota-baxter-Lie}
    R(x)\circ R(y)=R\Big(R(x)\circ y+ x\circ R(y)+\lambda x\circ y\Big)
\end{equation}
for all $x,y$ in $\mathcal{H}(\mathcal{A})$.
\begin{prop}
Let $(\mathcal{A},\circ,\alpha,\beta)$ be a BiHom-pre-Lie superalgebra such that $\alpha$ and $\beta$ are bijective.
\begin{enumerate}
\item Define $[~,~]_C:\mathcal{A}\otimes\mathcal{A}\rightarrow\mathcal{A}$ by $[x,y]_C=x\circ y-(-1)^{|x||y|}(\alpha^{-1}\beta(y))\circ(\alpha\beta^{-1}(x))$.
Then  $(\mathcal{A},[~~,~~]_C,\alpha,\beta)$ defines a BiHom-Lie superalgebra  which is denoted by $\mathcal{A}^C$ and called the sub-adjacent BiHom-Lie superalgebra of $\mathcal{A}$ and $\mathcal{A}$ is also called a compatible BiHom-pre-Lie superalgebra structure on the BiHom-Lie superalgebra.
\item The map  $L_\circ$ gives a representation of the BiHom-Lie superalgebra $(\mathcal{A},[~~,~~],\alpha,\beta)$, that is,
\begin{eqnarray*}
L_\circ(\alpha(x))\circ\alpha&=&\alpha\circ L_\circ(x),\\
L_\circ(\alpha(x))\circ\beta&=&\beta\circ L_\circ(x),\\
L_\circ([\beta(x),y])\circ\beta&=&L_\circ(\alpha\beta(x)) L_\circ(y)-(-1)^{|x||y|}L_\circ(\beta(y)) L_\circ(\alpha(x)).
\end{eqnarray*}
\end{enumerate}
\end{prop}
\begin{cor}Let $(\mathcal{A},\circ,\alpha,\beta)$ be a BiHom-pre-Lie superalgebra and $(V,l,r,\alpha_V,\beta_V)$ be an  $\mathcal{A}$-bimodule. Let $(\mathcal{A},[~~,~~]_C,\alpha,\beta)$ be the
subadjacent BiHom-Lie superalgebra. If $T$ is a super $\mathcal{O}$-operator associated to $(V,l,r,\alpha_V,\beta_V)$, then $T$ is a super $\mathcal{O}$-operator of $(\mathcal{A},[~~,~~]_C,\alpha,\beta)$ associated to $(V,l-r,r-l,\alpha_V,\beta_V)$.
\end{cor}
\begin{thm}Let $\mathcal{A}_{1}=(\mathcal{A},\circ,R,\alpha,\beta)$ be  a Rota-Baxter BiHom-pre-Lie superalgebra of weight zero. Then $\mathcal{A}_{2}=(\mathcal{A},\ast,R,\alpha,\beta)$ is a Rota-Baxter BiHom-pre-Lie superalgebra of weight zero, where the even binary operation $"\ast"$ is defined by
\begin{eqnarray*}
    x \ast y&=& R(x)\circ y-(-1)^{|x||y|}\alpha^{-1}\beta(y) \circ R(\alpha\beta^{-1}(x)).
\end{eqnarray*}
\end{thm}
\begin{proof} Since $\mathcal{A}_{1}=(\mathcal{A},\circ,\alpha,\beta)$ is a BiHom-pre-Lie superalgebra then we easily deduce that: \\

$\alpha(x\ast y)=\alpha(x)\ast \alpha(y)$ and $\beta(x\ast y)=\beta(x)\ast \beta(y)$.\\

Let $x,y$ and $z$ be a homogeneous elements in $\mathcal{A}$. Then we have
\begin{eqnarray*}
\alpha\beta(x)\ast (\alpha(y)\ast z)&=& R(\alpha\beta(x))\circ(\alpha(y)\ast z)-(-1)^{|x|(|y|+|z|)}(\beta(y)\ast\alpha^{-1}\beta(z))\circ R(\alpha^2(x))\\
&=& R(\alpha\beta(x))\circ(R(\alpha(y))\circ z)-(-1)^{|y||z|}R(\alpha\beta(x)\circ (\alpha^{-1}\beta(z)\circ R(\alpha^2\beta^{-1}(y)))\\
&-&(-1)^{|x|(|y|+|z|)}(R(\beta(y))\circ\alpha^{-1}\beta(z))\circ R(\alpha^2(x))\\
&+&(-1)^{|x|(|y|+|z|)+|y||z|}(\alpha^{-2}\beta^2(z)\circ R(\alpha(y)))\circ R(\alpha^2(x)),
\end{eqnarray*}
and
\begin{eqnarray*}
(\beta(x) \ast \alpha(y))\ast \beta(z)&=& R(\beta(x)\ast\alpha(y))\circ\beta(z)-(-1)^{|z|(|x|+|y|)}\alpha^{-1}\beta^2(z)\circ R(\alpha(x)\ast\alpha^2\beta^{-1}(y))\\
&=& R(R(\beta(x))\circ \alpha(y))\circ \beta(z)-(-1)^{|x||y|}R(\beta(y) \circ R(\alpha(x)))\circ \beta(z)\\
&-&(-1)^{|z|(|x|+|y|)}\alpha^{-1}\beta^2(z) \circ R(R(\alpha(x))\circ \alpha^2\beta^{-1}(y))\\
&+&(-1)^{|z|(|x|+|y|)+|x||y|}\alpha^{-1}\beta^2(z)\circ R(\alpha(y) \circ R(\alpha^2\beta^{-1}(x))).
\end{eqnarray*}%
Subtracting the above terms, switching $x$ and $y$, applied the fact that $R$ is Rota Baxter operator on $(\mathcal{A}_1,\circ,\alpha,\beta)$, and then subtracting the result yield:
\begin{small}
\begin{eqnarray*}
  && \ \    ass_{\mathcal{A}_{2}}(x,y,z)-(-1)^{|x||y|}ass_{\mathcal{A}_{2}}(y,x,z) \\
  && \ \ = \Big(\big(R(\beta(x))\circ R(\alpha(y))\big)\circ \beta(z)-R(\alpha\beta(x))\circ\big(R(\alpha(y))\circ z\big)\\
  && \ \  -(-1)^{|x||y|}\Big(R\big(\beta(y) \circ R(\alpha(x))\big)\circ\beta(z)+R(\alpha\beta(y)) \circ \big(\alpha(x)\circ z\big)\Big)\\
 && \ \  + (-1)^{|z|(|x|+|y|)}\Big[ \Big(-\alpha^{-1}\beta^2(z) \circ \big(R(\alpha(x))\circ R(\alpha^2\beta^{-1}(y))\big)+\big(\alpha^{-2}\beta^2(z) \circ R(\alpha(x))\big)\circ R(\alpha^2(y))\Big) \\
  && \ \  + (-1)^{|x||z|}\Big(R(\alpha\beta(x))\circ\big(\alpha^{-1}\beta(z) \circ R(\alpha^2\beta^{-1}(y))\big)-\big(R(\beta(x)) \circ R(\alpha^{-1}\beta(z))\big)\circ R(\alpha^2(y))\Big)\Big] \\
&& \ \  (-1)^{|z|(|x|+|y|)+|x||y|}\Big[ \Big(-\alpha^{-1}\beta^2(z) \circ \big(R(\alpha(y))\circ R(\alpha^2\beta^{-1}(x))\big)+\big(\alpha^{-2}\beta^2(z) \circ R(\alpha(y))\big)\circ R(\alpha^2(x))\Big) \\
  && \ \  + (-1)^{|y||z|}\Big(R(\alpha\beta(y))\circ\big(\alpha^{-1}\beta(z) \circ R(\alpha^2\beta^{-1}(x))\big)-\big(R(\beta(y)) \circ R(\alpha^{-1}\beta(z))\big)\circ R(\alpha^2(x))\Big)\Big] \\
  && \ \ \ = \Big(ass_{\mathcal{A}_{1}}\big(R(x),R(y),z\big)
  -(-1)^{|x||y|}ass_{\mathcal{A}_{1}}\big(R(y),R(x),z\big)\Big)\\
&& \ \  +(-1)^{|z|(|x|+|y|)}\Big(ass_{\mathcal{A}_{1}}\big(\alpha^{-2}\beta(z),R(x),R(\alpha^2\beta^{-1}(y))\big)
-(-1)^{|x||z|}ass_{\mathcal{A}_{1}}\big(R(x),\alpha^{-2}\beta(z),R(\alpha^2\beta^{-1}(y))\big) \Big)\\
  && \ \  -(-1)^{|z|(|x|+|y|)+|x||y|}\Big(ass_{\mathcal{A}_{1}}\big(\alpha^{-2}\beta(z),R(y),R(\alpha^2\beta^{-1}(x))\big)
-(-1)^{|y||z|}ass_{\mathcal{A}_{1}}\big(R(y),\alpha^{-2}\beta(z),R(\alpha^2\beta^{-1}(x))\big) \Big)\\
   && \ \  =0.
  \end{eqnarray*}
  \end{small}
Then $\mathcal{A}_{2}$ is a BiHom-pre-Lie superalgebra, which ends the proof.
\end{proof}
Now, we construct BiHom-pre-Lie superalgebras using super $\mathcal{O}$-operators on BiHom-Lie superalgebras.
\begin{prop}Let $(\mathcal{A},[~~,~~],\alpha,\beta)$ be a BiHom-Lie superalgebra and $(V,\rho,\alpha_V,\beta_V)$ be a representation of $\mathcal{A}$. Suppose that $T:V\longrightarrow \mathcal{A}$ is a super $\mathcal{O}$-operator of weight zero associated to $(V,\rho,\alpha_V,\beta_V)$. Then, the  even bilinear map
$$u\circ v=\rho(T(u))v,~~\forall~~u,v\in \mathcal{H}(V)$$
defines a BiHom-pre-Lie superalgebra structure on $V$.
\end{prop}
\begin{proof}Let $u,v$ and $w$ be  in $\mathcal{H}(V)$. We have
\begin{eqnarray*}
(\beta_V(u)\circ \alpha_V(v))\circ \beta_V(w)-\alpha_V\beta_V(u)\circ(\alpha_V(v)\circ w)
  =&& \rho\Big(T\big(\rho(T(\beta_V(u)))\alpha_V(v)\big)\Big)\beta_V(w)\\
-&&\rho\Big(T\big(\alpha_V\beta_V(u)\big)\Big)\rho\Big(T\big(\alpha_V(v)\big)\Big)w,
\end{eqnarray*}
\begin{eqnarray*}
(-1)^{|u||v|}\Big((\beta_V(v)\circ \alpha_V(u))\circ \beta_V(w)-\alpha_V\beta_V(v)\circ(\alpha_V(u)\circ w)\Big)
&=& (-1)^{|u||v|}
\Big[\rho\Big(T\big(\rho(T(\beta_V(v)))\alpha_V(u)\big)\Big)\beta_V(w)\\
&-&\rho\Big(T\big(\alpha_V\beta_V(v)\big)\Big)\rho\Big(T\big(\alpha_V(u)\big)\Big)w\Big]
\end{eqnarray*}
Hence \begin{eqnarray*}
      & & \ \ \ (\beta_V(u)\circ \alpha_V(v))\circ \beta_V(w)-\alpha_V\beta_V(u)\circ(\alpha_V(v)\circ w)
      -(-1)^{|u||v|}\Big((\beta_V(v)\circ \alpha_V(u))\circ \beta_V(w)-\alpha_V\beta_V(v)\circ(\alpha_V(u)\circ w)\Big)\\
      & & \ \ \  =\rho T\Big(\rho(T(\beta_V(u)))\alpha_V(v)-(-1)^{|u||v|}\rho(T(\beta_V(v)))\alpha_V(u)\Big)\beta_V(w)
      -\rho\Big(T\big(\alpha_V\beta_V(u)\big)\Big)\rho\Big(T\big(\alpha_V(v)\big)\Big)w\\
     & &\ \ \  +(-1)^{|u||v|}\rho\Big(T\big(\alpha_V\beta_V(v)\big)\Big)\rho\Big(T\big(\alpha_V(u)\big)\Big)w\\
      & & \ \ \  =\rho(\displaystyle\underbrace{[T(\beta_V(u)),T(\alpha_V(v))]}_{=[\beta(T(u)),T(\alpha_V(v))]})\beta_V(w)
      -\rho\Big(T\big(\alpha_V\beta_V(u)\big)\Big)\rho\Big(T\big(\alpha_V(v)\big)\Big)w
     +(-1)^{|u||v|}\rho\Big(T\big(\alpha_V\beta_V(v)\big)\Big)\rho\Big(T\big(\alpha_V(u)\big)\Big)w\\
      & & \ \ \  =\rho\Big(\alpha\beta(T(u))\Big)\rho\Big(T\big(\alpha_V(v)\big)\Big)w
     -(-1)^{|u||v|}\rho\Big(\beta(T\big(\alpha_V(v)\big)\Big)\rho\Big(\alpha(T(u))\Big)w\\
     & &\ \ \ -\rho\Big(T\big(\alpha_V\beta_V(u)\big)\Big)\rho\Big(T\big(\alpha_V(v)\big)\Big)w
     +(-1)^{|u||v|}\rho\Big(T\big(\alpha_V\beta_V(v)\big)\Big)\rho\Big(T\big(\alpha_V(u)\big)\Big)w\\
     & & \ \ \ =0
      \end{eqnarray*}
      It's since \eqref{bihom-lie-rep-1} - \eqref{bihom-lie-rep-3}.
      Therefore $(\mathcal{A},\circ,\alpha_V,\beta_V)$ is a BiHom pre-Lie superalgebra.
\end{proof}
\begin{rem}\label{rota-baxter=BiHomprelie}Let $(\mathcal{A},[~~,~~],\alpha,\beta)$ be a BiHom-Lie superalgebra and $R$ be the super $\mathcal{O}$-operator $($of weight zero$)$ associated to the adjoint representation $(\mathcal{A},ad,\alpha,\beta)$. Then the even binary operation given by
$x\circ y=[R(x),y]$, for all $x, y \in \mathcal{H}(\mathcal{A})$,
defines a BiHom pre-Lie superalgebra structure on $\mathcal{A}$.
\end{rem}
\begin{df}Let $(\mathcal{A},[~~,~~],R,\alpha,\beta)$ be a Rota-Baxter BiHom-Lie superalgebra  of weight zero. A Rota-Baxter operator on an $\mathcal{A}$-module $V$ $($relative to $R)$ is a map $R_V:V\longrightarrow V$ such that, for all $x \in \mathcal{H}(\mathcal{A})$ and $v \in \mathcal{H}(V)$,
\begin{eqnarray*}
    & &  [R(x),R_V(v)] = R_V \Big([R(x),v]+[x,R_V(v)]\Big),\\
    & & [R_V(v),R(x)] = R_V \Big([R_V(v),x]+[v,R(x)]\Big),
    \end{eqnarray*}
    where the action $\rho(x)(v)$ is denoted by $[x,v]$.
\end{df}
 The Lie-admissible algebras were studied by A. A. Albert in $1948$ and  M. Goze and E. Remm, in $2004$,  introduced the notion of $G$-associative algebras where $G$ is a subgroup of the permutation group $S_3$ $($see \cite{Goze}$)$. The graded case was studied by F. Ammar and A. Makhlouf in $2010$,  $($see \cite{Ammar-Makhlouf} for more details$)$. In \cite{Abdaoui-Mabrouk-Makhlouf} , the authors has been construct a functor from a full subcategory of the category of Rota-Baxter  Lie-admissible (or associative) superalgebras to the category of pre-Lie superalgebras. In this part, we study this construction in the BiHom case.
\begin{df}( \cite{WG-BiHom-Lie supealgebras} )\

 A BiHom-Lie admissible superalgebra is a BiHom superalgebra $(\mathcal{A},\mu,\alpha,\beta)$ in which the super-commutator
bracket, defined for all homogeneous $x,y$ in $\mathcal{A}$ by
$$[x,y] = \mu(x,y)-(-1)^{|x||y|}\mu(\alpha^{-1}\beta(y),\alpha\beta^{-1}(x)),$$
satisfies the BiHom super-Jacobi identity (\ref{B-H-sJ}).
\end{df}
\begin{df}\

Let $G$ be a subgroup of the permutation group $\mathcal{S}_3$. A Rota-Baxter $G$-BiHom-associative superalgebra of weight $\lambda \in \mathbb{K}$ is a $G$-BiHom-associative superalgebra $(\mathcal{A},\cdot,\alpha,\beta)$ together with an even linear self-map $R: \mathcal{A}\longrightarrow \mathcal{A}$
 that satisfies the identity
\begin{equation}\label{BiHom-iden-Rota}
    R(x) \cdot R(y) = R\big(R(x)\cdot y + x \cdot R(y) +\lambda x \cdot y\big),
\end{equation}
for all homogeneous elements $x,y,z$ in $\mathcal{A}$.
\end{df}
\begin{thm}
Let $(\mathcal{A},\cdot,R,\alpha,\beta)$ be a rota-Baxter BiHom-Lie admissible superalgebra of weight zero. The even binary operation "$\ast$" defined, for any homogeneous element $x,y\in\mathcal{A}$, by
$$x\ast y=[R(x),y].$$
Then $\mathcal{A}_L=(\mathcal{A},\ast,\alpha,\beta)$ is a BiHom-pre-Lie superalgebras.

\end{thm}
\begin{proof}
A direct consequence of Remark \eqref{rota-baxter=BiHomprelie}, since a Rota-Baxter operator on a BiHom-Lie admissible superalgebra is also a Rota-Baxter operatorof its supercommutator BiHom-lie superalgebra.
\end{proof}

\begin{thm}\label{RB+ass=BiHompreLie super} Let $(\mathcal{A},\cdot,R,\alpha,\beta)$ be a Rota-Baxter BiHom associative superalgebra of weight $\lambda=-1$. Define the even binary operation $"\circ"$ on any homogeneous element $x,y\in \mathcal{A}$, by
\begin{eqnarray}\label{ass+RB==pre-Lie}
x \circ y &=& R(x)\cdot y-(-1)^{|x||y|} \alpha^{-1}\beta(y) \cdot R(\alpha\beta^{-1}(x))-x\cdot y.
\end{eqnarray}
Then $\mathcal{A}_{L}=(\mathcal{A},\circ,\alpha,\beta)$ is a BiHom-pre-Lie superalgebra.
\end{thm}
\begin{proof}For all $x,y,z$ in $\mathcal{H}(\mathcal{A})$, we have
\begin{eqnarray*}
  (\beta(x)\circ \alpha(y))\circ \beta(z) &=& R\big(R(\beta(x))\cdot\alpha(y)\big)\cdot\beta(z)
  -(-1)^{|x||y|}R\big(\beta(y)\cdot R(\alpha(x))\big)\cdot\beta(z)-R\big(\beta(x)\cdot \alpha(y)\big)\cdot\beta(z)\\
  &-& (-1)^{|z|(|x|+|y|)}\alpha^{-1}\beta^2(z)\cdot R\big(R(\alpha(x))\cdot\alpha^2\beta^{-1}(y)\big) \\
   &+&(-1)^{|z|(|x|+|y|)+|x||y|}\alpha^{-1}\beta^2(z)\cdot R\big(\alpha(y)\cdot R(\alpha^2\beta^{-1}(x))\big)\\
   &+&(-1)^{|z|(|x|+|y|)}\alpha^{-1}\beta^2(z)\cdot R\big(\alpha(x)\cdot\alpha^2\beta^{-1}(y)\big)\\
   &-& \big(R(\beta(x))\cdot\alpha(y)\big)\cdot\beta(z)+(-1)^{|x||y|}\big(\beta(y)\cdot R(\alpha(x))\big)\cdot\beta(z)
    +\big(\beta(x)\cdot\alpha(y)\big)\cdot\beta(z).
  \end{eqnarray*}
   and
\begin{eqnarray*}
  \alpha\beta(x)\circ(\alpha(y)\circ z) &=& R\big(\alpha\beta(x)\big)\cdot\big(R(\alpha(y))\cdot z\big)
  -(-1)^{|y||z|}R\big(\alpha\beta(x)\big)\cdot\big(\alpha^{-1}\beta(z)\cdot R\big(\alpha^2\beta^{-1}(y)\big)\big)\\
  &-&R\big(\alpha\beta(x)\big)\cdot\big(\alpha(y)\cdot z\big)
   - (-1)^{|x|(|y|+|z|)}\big(R(\beta(y)\big)\cdot\alpha^{-1}\beta(z)\big)R(\alpha^2(x))\\
   &+&(-1)^{|x|(|y|+|z|)+|y||z|}\big(\alpha^{-2}\beta^2(z)\cdot R(\alpha(y)))\cdot R\big(\alpha^2(x)\big)\\
   &+&(-1)^{|x|(|y|+|z|)}\big(\beta(y)\cdot\alpha^{-1}\beta(z)\big)\cdot R\big(\alpha^2(x)\big)
   - \alpha\beta(x)\cdot\big(R(\alpha(y))\cdot z\big)\\&+&(-1)^{|y||z|}\alpha\beta(x)\cdot\big(\alpha^{-1}\beta(z)\cdot R(\alpha^2\beta^{-1}(y))\big)
   + \alpha\beta(x)\cdot\big(\alpha(y)\cdot z\big).
  \end{eqnarray*}
Then,  we obtain
\begin{eqnarray*}
  && as_{\mathcal{A}_{L}}(x,y,z)-(-1)^{|x||y|}as_{\mathcal{A}_{L}}(y,x,z) \\
   &=&(\beta(x)\circ \alpha(y))\circ \beta(z)-\alpha\beta(x)\circ(\alpha(y)\circ z)-(-1)^{|x||y|}(\beta(y)\circ \alpha(x))\circ \beta(z)
   +(-1)^{|x||y|}\alpha\beta(y)\circ(\alpha(x)\circ z)\\
   &=&  R\big(R(\beta(x))\cdot\alpha(y)\big)\cdot\beta(z)
  -(-1)^{|x||y|}R\big(\beta(y)\cdot R(\alpha(x))\big)\cdot\beta(z)-R\big(\beta(x)\cdot \alpha(y)\big)\cdot\beta(z)\\
  &-& (-1)^{|z|(|x|+|y|)}\alpha^{-1}\beta^2(z)\cdot R\big(R(\alpha(x))\cdot\alpha^2\beta^{-1}(y)\big) \\
   &+&(-1)^{|z|(|x|+|y|)+|x||y|}\alpha^{-1}\beta^2(z)\cdot R\big(\alpha(y)\cdot R(\alpha^2\beta^{-1}(x))\big)\\
   &+&(-1)^{|z|(|x|+|y|)}\alpha^{-1}\beta^2(z)\cdot R\big(\alpha(x)\cdot\alpha^2\beta^{-1}(y)\big)\\
   &-& \big(R(\beta(x))\cdot\alpha(y)\big)\cdot\beta(z)+(-1)^{|x||y|}\big(\beta(y)\cdot R(\alpha(x))\big)\cdot\beta(z)
    +\big(\beta(x)\cdot\alpha(y)\big)\cdot\beta(z)\\
    &-&R\big(\alpha\beta(x)\big)\cdot\big(R(\alpha(y))\cdot z\big)
  +(-1)^{|y||z|}R\big(\alpha\beta(x)\big)\cdot\big(\alpha^{-1}\beta(z)\cdot R\big(\alpha^2\beta^{-1}(y)\big)\big)\\
  &+&R\big(\alpha\beta(x)\big)\cdot\big(\alpha(y)\cdot z\big)
   + (-1)^{|x|(|y|+|z|)}\big(R(\beta(y)\big)\cdot\alpha^{-1}\beta(z)\big)R(\alpha^2(x))\\
   &-&(-1)^{|x|(|y|+|z|)+|y||z|}\big(\alpha^{-2}\beta^2(z)\cdot R(\alpha(y)))\cdot R\big(\alpha^2(x)\big)\\
   &-&(-1)^{|x|(|y|+|z|)}\big(\beta(y)\cdot\alpha^{-1}\beta(z)\big)\cdot R\big(\alpha^2(x)\big)
   + \alpha\beta(x)\cdot\big(R(\alpha(y))\cdot z\big)\\&-&(-1)^{|y||z|}\alpha\beta(x)\cdot\big(\alpha^{-1}\beta(z)\cdot R(\alpha^2\beta^{-1}(y))\big)
   - \alpha\beta(x)\cdot\big(\alpha(y)\cdot z\big)\\
   &-&(-1)^{|x||y|}R\big(R(\beta(y))\cdot\alpha(x)\big)\cdot\beta(z)
  +R\big(\beta(x)\cdot R(\alpha(y))\big)\cdot\beta(z)-R\big(\beta(y)\cdot \alpha(x)\big)\cdot\beta(z)\\
  &+& (-1)^{|z|(|x|+|y|)+|x||y|}\alpha^{-1}\beta^2(z)\cdot R\big(R(\alpha(y))\cdot\alpha^2\beta^{-1}(x)\big) \\
   &-&(-1)^{|z|(|x|+|y|)}\alpha^{-1}\beta^2(z)\cdot R\big(\alpha(x)\cdot R(\alpha^2\beta^{-1}(y))\big)\\
   &-&(-1)^{|z|(|x|+|y|)+|x||y|}\alpha^{-1}\beta^2(z)\cdot R\big(\alpha(y)\cdot\alpha^2\beta^{-1}(x)\big)\\
   &+&(-1)^{|x||y|} \big(R(\beta(y))\cdot\alpha(x)\big)\cdot\beta(z)-\big(\beta(x)\cdot R(\alpha(y))\big)\cdot\beta(z)
    -(-1)^{|x||y|}\big(\beta(y)\cdot\alpha(x)\big)\cdot\beta(z)\\
    &+&(-1)^{|x||y|}R\big(\alpha\beta(y)\big)\cdot\big(R(\alpha(x))\cdot z\big)
  -(-1)^{|y|(|x|+|z|)}R\big(\alpha\beta(y)\big)\cdot\big(\alpha^{-1}\beta(z)\cdot R\big(\alpha^2\beta^{-1}(x)\big)\big)\\
  &-&(-1)^{|x||y|}R\big(\alpha\beta(y)\big)\cdot\big(\alpha(x)\cdot z\big)
   - (-1)^{|x||z|}\big(R(\beta(x)\big)\cdot\alpha^{-1}\beta(z)\big)R(\alpha^2(y))\\
   &+&(-1)^{|z|(|x|+|y|)}\big(\alpha^{-2}\beta^2(z)\cdot R(\alpha(x)))\cdot R\big(\alpha^2(y)\big)\\
   &+&(-1)^{|x||z|}\big(\beta(x)\cdot\alpha^{-1}\beta(z)\big)\cdot R\big(\alpha^2(y)\big)
   -(-1)^{|x||y|} \alpha\beta(y)\cdot\big(R(\alpha(x))\cdot z\big)\\&-&(-1)^{|y||z|}\alpha\beta(y)\cdot\big(\alpha^{-1}\beta(z)\cdot R(\alpha^2\beta^{-1}(x))\big)
   +(-1)^{|x||y|} \alpha\beta(y)\cdot\big(\alpha(x)\cdot z\big)\\
   &=&\underbrace{\underbrace{\Big[R\big(R(\beta(x))\cdot\alpha(y)\big)+R\big(\beta(x)\cdot R(\alpha(y))\big)-R(\beta(x)\cdot\alpha(y))\Big]\cdot\beta(z)}_{=\big(R(\beta(x))\cdot R(\alpha(y))\big)\cdot\beta(z)}
   -R\big(\alpha\beta(x)\big)\cdot\big(R(\alpha(y))\cdot z\big)}_{=0}\\
\end{eqnarray*}
\begin{eqnarray*}
 &-&(-1)^{|x||y|}\Big[R\big(R(\beta(y))\cdot\alpha(x)\big)+R\big(\beta(y)\cdot R(\alpha(x))\big)-R(\beta(y)\cdot\alpha(x))\Big]\cdot\beta(z)
   +(-1)^{|x||y|}R\Big(\big(\alpha\beta(y)\big)\cdot\big(R(\alpha(x))\cdot z\big)\Big)\\
  &-&(-1)^{|z|(|x|+|y|)} \alpha^{-1}\beta^2(z)\cdot\Big[R\big(R(\alpha(x))\cdot\alpha^2\beta^{-1}(y)\big)+R\big(\alpha(x)\cdot R(\alpha^2\beta^{-1}(y))\big)-R(\alpha(x)\cdot\alpha^2\beta^{-1}(y))\Big]\\
  &+&(-1)^{|z|(|x|+|y|)}\Big(\big(\alpha^{-2}\beta^2(z)\cdot R(\alpha(x))\big)\cdot R(\alpha^2(y))\Big)\\
   &+&(-1)^{|z|(|x|+|y|)+|x||y|} \alpha^{-1}\beta^2(z)\cdot\Big[R\big(R(\alpha(y))\cdot\alpha^2\beta^{-1}(x)\big)+R\big(\alpha(y)\cdot R(\alpha^2\beta^{-1}(x))\big)-R(\alpha(y)\cdot\alpha^2\beta^{-1}(x))\Big]\\
   &-&(-1)^{|z|(|x|+|y|)+|x||y|}\Big(\big(\alpha^{-2}\beta^2(z)\cdot R(\alpha(y))\big)\cdot R(\alpha^2(x))\Big)
   -\Big[\big(R(\beta(x))\cdot\alpha(y)\big)\cdot\beta(z)-r(\alpha\beta(x))\cdot\big(\alpha(y)\cdot z\big)\Big]\\
   &+&(-1)^{|x||y|}\Big[\big(R(\beta(y))\cdot\alpha(x)\big)\cdot\beta(z)-r(\alpha\beta(y))\cdot\big(\alpha(x)\cdot z\big)\Big]
   +\Big[\alpha\beta(x)\cdot\big(R(\alpha(y))\cdot z\big)-\big(\beta(x)\cdot R(\alpha(y))\big)\cdot\beta(z)\Big]\\
   &+&(-1)^{|x||y|}\Big[\big(\beta(y)\cdot R(\alpha(x))\big)\cdot\beta(z)-\alpha\beta(y)\cdot\big(R(\alpha(x))\cdot z\big)\Big]
   +\Big[\big(\beta(x)\cdot\alpha(y)\big)\cdot\beta(z)-\alpha\beta(x)\cdot\big(\alpha(y)\cdot z\big)\Big]\\
   &-&(-1)^{|x||y|}\Big[\big(\beta(y)\cdot\alpha(x)\big)\cdot\beta(z)-\alpha\beta(y)\cdot\big(\alpha(x)\cdot z\big)\Big]\\
   &-&(-1)^{|y||z|}\Big[R(\alpha\beta(x))\cdot\big(\alpha^{-1}\beta(z)\cdot R(\alpha^2\beta^{-1}(y))\big)-
   \big(R(\beta(x))\cdot\alpha^{-1}\beta(z)\big)\cdot R(\alpha^2(y))\Big]\\
   &-&(-1)^{|x|(|y|+|z|)}\Big[R(\alpha\beta(y))\cdot\big(\alpha^{-1}\beta(z)\cdot R(\alpha^2\beta^{-1}(x))\big)-
   \big(R(\beta(y))\cdot\alpha^{-1}\beta(z)\big)\cdot R(\alpha^2(x))\Big]\\
   &-&(-1)^{|y||z|}\Big[\alpha\beta(x)\cdot\big(\alpha^{-1}\beta(z)\cdot R(\alpha^2\beta^{-1}(y))\big)-
   \big(\beta(x)\cdot\alpha^{-1}\beta(z)\big)\cdot R(\alpha^2(y))\Big]\\
   &-&(-1)^{|x|(|y|+|z|)}\Big[\alpha\beta(y)\cdot\big(\alpha^{-1}\beta(z)\cdot R(\alpha^2\beta^{-1}(x))\big)-
   \big(\beta(y)\cdot\alpha^{-1}\beta(z)\big)\cdot R(\alpha^2(x))\Big]\\
   &=&0.
   \end{eqnarray*}
It is since the associativity and the Rota-Baxter identity (\ref{BiHom-iden-Rota}) with $ \lambda=-1$.
\end{proof}
\begin{cor}\label{rota-BiHo-ass-pre-Lie} Let $(\mathcal{A},\cdot,R,\alpha,\beta)$ be a Rota-Baxter BiHom associative superalgebra of weight $\lambda=-1$. Then $R$ is still a Rota-Baxter operator of weight $\lambda=-1$ on the BiHom-pre-Lie superalgebra $(\mathcal{A},\circ,\alpha,\beta)$ defined in  (\ref{ass+RB==pre-Lie}).
\end{cor}

As a consequence of Theorem \ref{RB+ass=BiHompreLie super} and Corollary \ref{rota-BiHo-ass-pre-Lie}, we have
\begin{prop} Let $(\mathcal{A},\cdot,R,\alpha,\beta)$ be a Rota-Baxter BiHom associative superalgebra of weight $\lambda=-1$. Then the binary operation defined, for any homogeneous elements $x,\;y$ in $\mathcal{A}$, by
\begin{eqnarray*}
[x,y]&=&R(x)\cdot y-(-1)^{|x||y|}\alpha^{-1}\beta(y)\cdot R(\alpha\beta^{-1}(x))-x\cdot y+x\cdot R(y)\\&-&(-1)^{|x||y|}R(\alpha^{-1}\beta(y))\cdot\alpha\beta^{-1}(x)+(1)^{|x||y|}\alpha^{-1}\beta(y)\cdot \alpha\beta^{-1}(x),
\end{eqnarray*}
defines a Rota-Baxter Lie superalgebra $(\mathcal{A},[~,~],R,\alpha,\beta)$ of weight $\lambda=-1$.
\end{prop}

\section{BiHom-$L$-Dendriform Superalgebras}
The notion of $L$-dendriform algebra was introduced by C. Bai, L. Liu and X. Ni in $2010$  $($see \cite{Bai-Liu L-dendrifom}$)$. The Hom case was introduced by Ibrahima Bakayoko $($see \cite{IbrahimaBakayoko}$)$. In this section, we extend this notion to the graded case of BiHom-$L$-dendriform superalgebra. Then we study  relationships between  associative BiHom superalgebras, BiHom-$L$-dendriform superalgebras and BiHom-pre-Lie superalgebras.  Moreover, we introduce the notion of Rota-Baxter operator $($of weight zero$)$ on the $\mathcal{A}$-bimodule and we provide construction of associative bimodules from bimodules over BiHom-$L$-dendriform superalgebras.
\subsection{BiHom-$L$-Dendriform  Superalgebras and BiHom associative Superalgebras}
\subsubsection{Definition and Some Basic Properties}

\begin{df}A BiHom-$L$-dendriform superalgebra is a $5$-tuple $(\mathcal{A},\rhd,\lhd,\alpha,\beta)$ consisting of a $\mathbb{Z}_2$-graded vector space $\mathcal{A}$,  two even bilinear maps
$\rhd,\lhd:\mathcal{A}\otimes \mathcal{A}\longrightarrow\mathcal{A}$ and two commuting even linear maps $\alpha,\beta  : \mathcal{A} \longrightarrow \mathcal{A}$ satisfying the following conditions (for all $x, y, z\in\mathcal{H}(\mathcal{A})$):
\begin{eqnarray}\label{alpha-multip-dend}
&&\alpha(x \rhd y) = \alpha(x)\rhd\alpha(y)\;,\; \alpha(x \lhd y) = \alpha(x) \lhd \alpha(y),\\
\label{beta-multip-dend} &&\beta(x \rhd y) = \beta(x)\rhd\beta(y)\;,\; \beta(x \lhd y) = \beta(x) \lhd \beta(y),
\end{eqnarray}
\small{\begin{eqnarray}\label{1-Ldend}
\alpha\beta(x)\rhd(\alpha(y)\rhd z)&=&(\beta(x)\rhd \alpha(y))\rhd \beta(z)+(\beta(x)\lhd \alpha(y))\rhd \beta(z)
+(-1)^{|x||y|}\alpha\beta(y)\rhd (\alpha(x)\rhd z)\\
&-&(-1)^{|x||y|}(\beta(y)\lhd \alpha(x))\rhd \beta(z)-(-1)^{|x||y|}(\beta(y)\rhd \alpha(x))\rhd \beta(z),\nonumber
 \end{eqnarray}}
 \begin{eqnarray}\label{2-Ldend}
\alpha\beta(x)\rhd(\alpha(y)\lhd z)&=&(\beta(x)\rhd \alpha(y))\lhd \beta(z)+(-1)^{|x||y|}\alpha\beta(y)\lhd (\alpha(x)\rhd z)\\
&+&(-1)^{|x||y|}\alpha\beta(y)\lhd (\alpha(x) \lhd z)-(-1)^{|x||y|}(\beta(y)\lhd \alpha(x))\lhd \beta(z).\nonumber
 \end{eqnarray}
\end{df}
 \begin{prop}Let $(\mathcal{A},\rhd,\lhd,\alpha,\beta)$ be a BiHom-$L$-dendriform superalgebra.
\begin{enumerate}
\item The even binary operation $\circ:\mathcal{A}\otimes \mathcal{A}\longrightarrow \mathcal{A}$ given for all $x,y \in \mathcal{H}(\mathcal{A})$ by
\begin{eqnarray}
\label{BiHom-pre-L-dend-superalg1} x\circ y&=&x\triangleright y-(-1)^{|x||y|}\alpha^{-1}\beta(y)\triangleleft\alpha\beta^{-1}(x),
\end{eqnarray}
 defines a BiHom-pre-Lie superalgebra $(\mathcal{A},\circ,\alpha,\beta)$ which is called the associated vertical BiHom-pre-Lie superalgebra of $(\mathcal{A},\rhd,\lhd,\alpha,\beta)$ and $(\mathcal{A},\rhd,\lhd,\alpha,\beta)$ is called a compatible BiHom-L-dendriform superalgebra structure on the BiHom-pre-Lie superalgebra $(\mathcal{A},\circ,\alpha,\beta)$.
\item The even binary operation $\bullet:\mathcal{A}\otimes \mathcal{A}\longrightarrow \mathcal{A}$ given by
\begin{eqnarray}
\label{BiHom-pre-L-dend-superalg2} x\bullet y=x\rhd y+x \lhd y,~~\forall~~x,y \in \mathcal{H}(\mathcal{A})
 \end{eqnarray}
defines a BiHom-pre-Lie superalgebra $(\mathcal{A},\bullet,\alpha,\beta)$ which is called the associated horizontal BiHom-pre-Lie superalgebra of $(\mathcal{A},\rhd,\lhd,\alpha,\beta)$ and $(\mathcal{A},\rhd,\lhd,\alpha,\beta)$ is called a compatible BiHom-L-dendriform superalgebra structure on the BiHom-pre-Lie superalgebra $(\mathcal{A},\bullet,\alpha,\beta)$.
\end{enumerate}
\end{prop}
 \begin{rem}
 If $\alpha=\beta=Id$, then $(\mathcal{A},\circ)$ and $(\mathcal{A},\bullet)$ are a pre-Lie superalgebras.
 \end{rem}
\begin{cor}
Let $(\mathcal{A},\rhd,\lhd,\alpha,\beta)$ be a BiHom-$L$-dendriform superalgebra. Then the brackets
\begin{eqnarray}\label{BiHom-Lie-L-dend-sup-alg1}
[x,y]&=& x\circ y-(-1)^{|x||y|}\alpha^{-1}\beta(y)\circ \alpha\beta^{-1}(x)\\
&=& x\triangleright y-(-1)^{|x||y|}\alpha^{-1}\beta(y)\triangleleft\alpha\beta^{-1}(x)-(-1)^{|x||y|}\alpha^{-1}\beta(y)\triangleright \alpha\beta^{-1}(x)+x\triangleleft y
\nonumber,
\end{eqnarray}
define a BiHom-Lie superalgebra structure on $\mathcal{A}$.
\end{cor}
\begin{rem}
\begin{eqnarray*}
\{x,y\}&=&x\bullet y-(-1)^{|x||y|}\alpha^{-1}\beta(y)\bullet\alpha\beta^{-1}(x)\\
&=&x\rhd y-(-1)^{|x||y|}\alpha^{-1}\beta(y)\lhd\alpha\beta^{-1}(x)-(-1)^{|x||y|}\alpha^{-1}\beta(y)\rhd\alpha\beta^{-1}(x)+x \lhd y\\
&=&[x,y].
\end{eqnarray*}
Then $\{,\}$ is also defined a BiHom-Lie superalgebra structure on $\mathcal{A}$.
\end{rem}

The below result allows to obtain a BiHom-L-dendriform superalgebra from a given one by transposition.

\begin{prop}
Let $(\mathcal{A},\triangleright,\triangleleft, \alpha,\beta)$ be a BiHom-L-dendriform superalgebra. Define two even binary operations
$\triangleright^t, \triangleleft^t : A \otimes A \longrightarrow A$ by
\begin{equation}\label{BiHom-transpose-L-dend}
x \triangleright^t y := x \triangleright y,\;\;\;\; x \triangleleft^t  y := -(-1)^{|x||y|}\alpha^{-1}\beta(y) \triangleleft\alpha\beta^{-1}(y)(x)
\end{equation}
Then $(A, \triangleright^t, \triangleleft^t, a)$ is a BiHom-L-dendriform superalgebra, and $\circ^t =\bullet$ and $\bullet^t =\circ$.
The BiHom-L-dendriform superalgebra $(A, \triangleright^t, \triangleleft^t,\alpha,\beta)$ is called the transpose of $(A,\triangleright,\triangleleft,\alpha,\beta)$.
\begin{proof}

\end{proof}

\end{prop}

\begin{df}\
\begin{enumerate}
\item Let $(\mathcal{A},\rhd,\lhd,\alpha,\beta)$ be a BiHom-$L$-dendriform superalgebra, $V$ be a $\mathbb{Z}_2$-graded vector space, $l_\rhd,r_\rhd,l_\lhd,r_\lhd:\mathcal{A}\longrightarrow End(V)$ be four even linear maps and $\alpha_V,\beta_V:V\longrightarrow V$ be two even linear maps. The tuple $(V,l_\rhd,r_\rhd,l_\lhd,r_\lhd,\alpha_V,\beta_V)$ is an $\mathcal{A}$-bimodule if for any homogeneous elements $x,y\in \mathcal{A}$,  the following identities are satisfied
\begin{enumerate}
\item  $ l_\rhd([\beta(x),\alpha(y)])\beta_V=l_\rhd(\alpha\beta(x))l_\rhd(\alpha(y))-(-1)^{|x||y|}l_\rhd(\alpha\beta(y))l_\rhd(\alpha(x)),$
\item    $ l_\lhd(\beta(x) \circ \alpha(y))\beta_V=l_\rhd(\alpha\beta(x))l_\lhd(\alpha(y))-(-1)^{|x||y|}l_\lhd(\alpha\beta(y))l_\lhd(\alpha(x))
-(-1)^{|x||y|}l_\lhd(\alpha\beta(y))l_\rhd(\alpha(x)),$
\item
 $r_\rhd(\alpha(x)\rhd y)\alpha_V\beta_V=r_\rhd(\beta(y))r_\rhd(\alpha(x))\beta_V+r_\rhd(\beta(y))r_\lhd(\alpha(x))\beta_V+(-1)^{|x||y|}l_\rhd(\alpha\beta(x))r_\rhd(y)\alpha_V\\
 \;\;\;\;\;\;\;\;\;\;\;-(-1)^{|x||y|}r_\rhd(\beta(y))l_\rhd(\beta(x))\alpha_V-(-1)^{|x||y|}r_\rhd(\beta(y))l_\lhd(\beta(x))\alpha_V$,
\item   $  r_\rhd(\alpha(x)\lhd y)\alpha_V\beta_V= r_\lhd(\beta(y))r_\rhd(\alpha(x))\beta_V+(-1)^{|x||y|}l_\lhd(\alpha\beta(x))r_\rhd(y)\alpha_V\\+(-1)^{|x||y|}l_\lhd(\alpha\beta(x))r_\lhd(y)\alpha_V
    -(-1)^{|x||y|}r_\lhd(\beta(y))l_\lhd(\beta(x))\alpha_V,$
\item    $ r_\lhd(\alpha(x)\bullet y)\alpha_V\beta_V=r_\lhd(\beta(y))r_\lhd(\alpha(x))\beta_V-(-1)^{|x||y|} r_\lhd(\beta(y))l_\rhd(\beta(x))\alpha_V+(-1)^{|x||y|}l_\rhd(\alpha\beta(x))r_\lhd(y)\alpha_V.$
\end{enumerate}
Moreover, The tuple $(V,\rhd_V,\lhd_V,l_\rhd,r_\rhd,l_\lhd,r_\lhd,\alpha_V,\beta_V)$ is an $\mathcal{A}$-bimodule $\K$-superalgebra if  the following identities are satisfied
\begin{enumerate}
\item    $ \alpha_V\beta_V(u)\rhd_Vl_\rhd(\alpha(x))(v)-(-1)^{|x||u|}l_\rhd(\alpha\beta(x))(\alpha_V(u)\rhd_V v)=
r_\rhd(\alpha(x))(\beta_V(u))\rhd_V\beta_V(v)$\\
 $  \hspace{3cm} +r_\lhd(\alpha(x))(\beta_V(u))\rhd_V\beta_V(v)-(-1)^{|x||u|}l_\lhd(\beta(x))(\alpha_V(u)\rhd_V \beta_V(v))$\\
$ \hspace{3cm}-(-1)^{|x||u|}l_\rhd(\beta(x))(\alpha_V(u)\rhd_V \beta_V(v)),$
\item    $ \alpha_V\beta_V(u)\rhd_Vr_\rhd(x)(\alpha_V(v))-(-1)^{|x||u|}\alpha_V\beta_V(v)\rhd_Vr_\rhd(x)(\alpha_V(u))=
r_\rhd(\beta(x))(\beta_V(u)\bullet_V\alpha_V(v))-(-1)^{|x||u|}r_\rhd(\beta(x))(\beta_V(v)\bullet_V\alpha_V(u)),$
\item    $l_\lhd(\alpha\beta(x))(\alpha_V(u)\lhd_v v)=l_\rhd(\beta(x))(\alpha_V(u))\lhd_V \beta_V(v)
+(-1)^{|x||u|}\alpha_V\beta_V(u)\lhd_Vl_\rhd(\alpha(x))(v)+(-1)^{|x||u|} \alpha_V\beta_V(u)\lhd_Vl_\lhd(\alpha(x))(v)
 -(-1)^{|x||u|}r_\lhd(\alpha(x))(\beta_V(u))\lhd_V\beta_V(v).$
 \item    $\alpha_V\beta_V(u)\lhd_Vl_\lhd(\alpha(x))(v)=r_\rhd(\alpha(x))(\beta_V(u))\lhd_V \beta_V(v)
+(-1)^{|x||u|}l_\lhd(\alpha\beta(x))(\alpha_V(u)\rhd_V v)+(-1)^{|x||u|}l_\lhd(\alpha\beta(x))(\alpha_V(u)\lhd_v v)
 -(-1)^{|x||u|}l_\lhd(\beta(x))(\alpha_V(u))\lhd_V\beta_V(v).$
\end{enumerate}

\item Let $(\mathcal{A},\rhd,\lhd,\alpha,\beta)$ be a BiHom-$L$-dendriform superalgebra and $(V,\rhd_V,\lhd_V,l_\rhd,r_\rhd,l_\lhd,r_\lhd,\alpha_V,\beta_V)$ be an $\mathcal{A}$-bimodule $\K$-superalgebra. An even linear map
$T:V\longrightarrow \mathcal{A}$ is called a super $\mathcal{O}$-operator of weight $\lambda \in \mathbb{K}$ associated to $(V,\rhd_V,\lhd_V,l_\rhd,r_\rhd,l_\lhd,r_\lhd,\alpha_V,\beta_V)$ if $T$ satisfies for any homogeneous elements $u,v$ in $V$
\begin{eqnarray*}
\alpha\circ T&=&T\circ\alpha_V~~~~~ and~~~~ \beta\circ T=T\circ\beta_V,\\
   T(u)\rhd T(v)&=& T\Big(l_\rhd(T(u))v+(-1)^{|u||v|}r_\rhd(T(v))u+\lambda u\rhd_V v \Big),\\
    T(u)\lhd T(v)&=& T\Big(l_\lhd(T(u))v+(-1)^{|u||v|}r_\lhd(T(v))u +\lambda u\lhd_V v\Big).
\end{eqnarray*}
\end{enumerate}
In particular, a super $\mathcal{O}$-operator of weight $\lambda \in \mathbb{K}$ of the BiHom-$L$-dendriform superalgebra $(\mathcal{A},\rhd,\lhd,\alpha,\beta)$
associated to the bimodule $(\mathcal{A},l_\rhd,r_\rhd,l_\lhd,r_\lhd,\alpha_V,\beta_V)$ is called a Rota-Baxter operator $($of weight $\lambda)$ on $(\mathcal{A},\rhd,\lhd,\alpha,\beta)$, that is, $R$ satisfies for any homogeneous elements $x,y$ in $\mathcal{A}$
\begin{eqnarray*}
\alpha\circ R&=&R\circ\alpha~~~~~ and~~~~ \beta\circ R=R\circ\beta,\\
   R(x)\rhd R(y)&=& R\Big(R(x)\rhd y+x\rhd R(y)+\lambda x\rhd y\Big),\\
    R(x)\lhd R(y)&=& R\Big(R(x)\lhd y+x\lhd R(y)+ \lambda x\lhd y\Big).
\end{eqnarray*}
\end{df}
The following theorem provides a construction of BiHom-$L$-dendriform superalgebras using super $\mathcal{O}$-operators of BiHom associative superalgebras.
\begin{thm}\label{ass=L-dendriform} Let $(\mathcal{A},\mu,\alpha,\beta)$ be a BiHom associative superalgebra and $(V,l,r,\alpha_V,\beta_V)$ be a $\mathcal{A}$-bimodule. If $T$ is a super $\mathcal{O}$-operator of weight zero associated with $(V,l,r,\alpha_V,\beta_V)$, then there exists a BiHom-$L$-dendriform superalgebra structure on $V$ defined by
\begin{equation}\label{th-operator1-ass}
    u \rhd v=l(T(u))v,~~u \lhd v= (-1)^{|u||v|}r(T(v))u,~~\forall~~u,v \in \mathcal{H}(V).
\end{equation}
\end{thm}
\begin{proof}By a direct computation,For any homogeneous elements $u,v$ and $w$ in $V$, we have\\
\begin{eqnarray*}
&& (\beta_V(u)\rhd \alpha_V(v))\rhd \beta_V(w)+(\beta_V(u)\lhd \alpha_V(v))\rhd \beta_V(w)-\alpha_V\beta_V(u)\rhd(\alpha_V(v)\rhd w)\\
&& +(-1)^{|u||v|}\alpha_V\beta_V(v)\rhd (\alpha_V(u)\rhd w)-(-1)^{|u||v|}(\beta_V(v)\lhd \alpha_V(u))\rhd \beta_V(w)
-(-1)^{|u||v|}(\beta_V(v)\rhd \alpha_V(u))\rhd \beta_V(w)\\
&& =l\Big(T\big(l(T(\beta_V(u)))\big)\alpha_V(v)\Big)\beta_V(w)+(-1)^{|u||v|}l\Big(T\big(r(T(\alpha_V(v)))\big)\beta_V(u)\Big)\beta_V(w)\\
&&-l\Big(T\big(\alpha_V\beta_V(u)\big)\Big)l\Big(T\big(\alpha_V(v)\big)\Big)w
+(-1)^{|u||v|}l\Big(T\big(\alpha_V\beta_V(v)\big)\Big)l\Big(T\big(\alpha_V(u)\big)\Big)w\\
&& -l\Big(T\big(r(T(\alpha_V(u)))\big)\beta_V(v)\Big)\beta_V(w)-(-)^{|u||v|}l\Big(T\big(l(T(\beta_V(v)))\big)\alpha_V(u)\Big)\beta_V(w)\\
&& =0,
\end{eqnarray*}
and similarly, we have
\begin{eqnarray*}
&& \alpha\beta(x)\rhd(\alpha(y)\lhd z)-(\beta(x)\rhd \alpha(y))\lhd \beta(z)-(-1)^{|x||y|}\alpha\beta(y)\lhd (\alpha(x)\rhd z)\\
&& -(-1)^{|x||y|}\alpha\beta(y)\lhd (\alpha(x) \lhd z)+(-1)^{|x||y|}(\beta(y)\lhd \alpha(x))\lhd \beta(z)=0
\end{eqnarray*}
Therefore $(V,\rhd,\lhd,\alpha_V,\beta_V)$ is a BiHom-$L$-dendriform superalgebra.
\end{proof}
A direct consequence of Theorem \ref{ass=L-dendriform} is the following construction of a BiHom-$L$-dendriform superalgebra from a Rota-Baxter operator $($of weight zero$)$ of a BiHom associative superalgebra.
\begin{cor}Let $(\mathcal{A},\mu,\alpha,\beta,R)$ be a Rota-Baxter BiHom associative superalgebra of weight zero. Then, the even binary operations given by
$$ x\rhd y=\mu(R(x),y),~~~~x\lhd y=\mu(x,R(y)),~~~~\forall~~x,y\in \mathcal{H}(\mathcal{A})$$
defines a BiHom-$L$-dendriform superalgebra structure on $\mathcal{A}$.
\end{cor}
\begin{df}Let $(\mathcal{A},\rhd,\lhd,\alpha,\beta)$ be a BiHom-$L$-dendriform superalgebra and $R:\mathcal{A}\longrightarrow \mathcal{A}$ be a Rota-Baxter operator of weight zero. A Rota-Baxter operator on $\mathcal{A}$-bimodule $V$ $($relative to $R)$ is an even map $R_V:V\longrightarrow V$ such that for all
homogeneous elements $x$ in $\mathcal{A}$ and $v$ in $V$
\begin{eqnarray*}
    & & \ \ \ R(x)\rhd R_V(v) = R_V \Big(R(x)\rhd v+x \rhd R_V(v)\Big), \ \
     R_V(v)\rhd R(x) = R_V \Big(R_V(v)\rhd x+v \rhd R(x)\Big),\\
    & & \ \ \ R(x)\lhd R_V(v) = R_V \Big(R(x)\lhd v+x \lhd R_V(v)\Big), \ \
    R_V(v)\lhd R(x) = R_V \Big(R_V(v)\lhd x+v \lhd R(x)\Big).
    \end{eqnarray*}
\end{df}
\begin{prop}Let $(\mathcal{A},\mu,\alpha,\beta)$ be a BiHom associative superalgebra, $R:\mathcal{A}\longrightarrow \mathcal{A}$ a Rota-Baxter operator on $\mathcal{A}$, $V$  an $\mathcal{A}$-bimodule and $R_V$ a Rota-Baxter operator on $V$. Define a new actions of $\mathcal{A}$ on $V$ by
$$ x\rhd v=\mu(R(x),v),~~v\rhd x=\mu(R_V(v),x),~~x\lhd v=\mu(x,R_V(v)),~~v\lhd x=\mu(v,R(x)).$$
Equipped with these actions, $V$ becomes  an $\mathcal{A}$-bimodule over the associated  BiHom-$L$-dendriform superalgebra.
\end{prop}
\begin{cor}Let $(V,l_\rhd,r_\rhd,l_\lhd,r_\lhd,\alpha_V,\beta_V)$ be an $ \mathcal{A}$-bimodule of a BiHom-$L$-dendriform superalgebra $(\mathcal{A},\rhd,\lhd,\alpha,\beta)$. Let $(\mathcal{A},\mu,\alpha,\beta)$ be the
associated BiHom associative superalgebra. If $T$ is a super $\mathcal{O}$-operator associated to $(V,l_\rhd,r_\rhd,l_\lhd,r_\lhd,\alpha_V,\beta_V)$,
then $T$ is a super $\mathcal{O}$-operator of $(\mathcal{A},\mu,\alpha,\beta)$ associated to $(V,l_\rhd+l_\lhd,r_\rhd+r_\lhd,\alpha_V,\beta_V)$.
\end{cor}
\subsection{BiHom-$L$-dendriform superalgebras and BiHom-pre-Lie superalgebras}\
\vskip 0.5cm

 Conversely, we can construct BiHom-$L$-dendriform superalgebras from $\mathcal{O}$-operators of BiHom-pre-Lie superalgebras.
\begin{thm}\label{preLie=L-dendriform}Let $(\mathcal{A},\circ,\alpha,\beta)$ be a BiHom-pre-Lie superalgebra and $(V,l,r,\alpha_V,\beta_V)$ be an $\mathcal{A}$-bimodule. If $T$ is a super $\mathcal{O}$-operator of weight zero associated to $(V,l,r,\alpha_V,\beta_V)$, then there exists a BiHom-$L$-dendriform superalgebra structure on $V$ defined by
\begin{equation}\label{th-operator1}
    u \rhd v=l(T(u))v,~~u \lhd v= -r(T(u))v,~~\forall~~u,v \in \mathcal{H}(V).
\end{equation}
Therefore, there is a BiHom-pre-Lie superalgebra structure on $V$ defined by
\begin{equation}\label{th-operator2}
    u \circ v= u \rhd v-(-1)^{|u||v|}\alpha_V^{-1}\beta_V(v) \lhd \alpha_V\beta_V^{-1}(u),~~\forall~~u,v \in \mathcal{H}(V)
\end{equation}
as the associated vertical BiHom-pre-Lie superalgebra of $(V,\rhd,\lhd,\alpha_V,\beta_V)$ and $T$ is a homomorphism of BiHom-pre-Lie superalgebra.\\
Furthermore, $T(V)=\{T(v)~~/~~ v\in V\}\subset \mathcal{A}$ is a BiHom-pre-Lie subsuperalgebra of $(\mathcal{A},\circ,\alpha,\beta)$ and there is a BiHom-$L$-dendriform superalgebra structure on $T(V)$ given by
\begin{equation}\label{th-operator3}
    T(u)\rhd T(v)= T(u \rhd v),~~ T(u)\lhd T(v)= T(u \lhd v),~~\forall~~u,v \in \mathcal{H}(V).
\end{equation}
Moreover, the corresponding associated vertical BiHom-pre-Lie superalgebra structure on $T(V)$ is a BiHom-pre-Lie subsuperalgebra of
$(\mathcal{A},\circ,\alpha,\beta)$ and $T$ is an homomorphism of BiHom-$L$-dendriform superalgebra.
\end{thm}
\begin{proof}For any homogeneous elements $u,v$ and $w$ in $V$, we have\\
$~~\alpha_V\beta_V(u)\rhd (\alpha_V(v)\rhd w)= l(T(\alpha_V\beta_V(u)))l(T(\alpha_V(v)))w$,\\
$~~(\beta_V(u)\rhd \alpha_V(v))\rhd\beta_V(w)= l\Big(T(l(T(\beta_V(u)))\alpha_V(v))\Big)\beta_V(w)$,\\
$~~(\beta_V(u)\lhd\alpha_V(v))\rhd\beta_V(w)=-l\Big(T(r(T(\beta_V(u)))\alpha_V(v))\Big)\beta_V(w),$\\
$~~(-1)^{|u||v|}\alpha_V\beta_V(v)\rhd(\alpha_V(u)\rhd w)= (-1)^{|u||v|}l(T(\alpha_V\beta_V(v)))l(T(\alpha_V(u)))w$,\\
$~~(-1)^{|u||v|}(\beta_V(v)\lhd \alpha_V(u))\rhd\beta_V(w)=-(-1)^{|u||v|}l\Big(T(r(T(\beta_V(v)))\alpha_V(u))\Big)\beta_V(w),$\\
$~~(-1)^{|u||v|}(\beta_V(v)\rhd \alpha_V(u))\rhd\beta_V(w)=(-1)^{|u||v|}l\Big(T(l(T(\beta_V(v)))\alpha_V(u))\Big)\beta_V(w).$\\
Hence
\begin{align*}
&    \alpha_V\beta_V(u)\rhd (\alpha_V(v)\rhd w) -(\beta_V(u)\rhd \alpha_V(v))\rhd\beta_V(w)-(\beta_V(u)\lhd\alpha_V(v))\rhd\beta_V(w)\\& -(-1)^{|u||v|}\alpha_V\beta_V(v)\rhd(\alpha_V(u)\rhd w)
+(-1)^{|u||v|}(\beta_V(u)\lhd \alpha_V(v))\rhd\beta_V(w)+(-1)^{|u||v|}(\beta_V(v)\rhd \alpha_V(u))\rhd\beta_V(w)\\
&    =l(T(\alpha_V\beta_V(u)))l(T(\alpha_V(v)))w -(-1)^{|u||v|}l(T(\alpha_V\beta_V(v)))l(T(\alpha_V(u)))w\\
& -l\Big(T(l(T(\beta_V(u)))\alpha_V(v))\Big)\beta_V(w)-(-1)^{|u||v|}l\Big(T(r(T(\beta_V(v)))\alpha_V(u))\Big)\beta_V(w)\\
& +(-1)^{|u||v|}l\Big(T(l(T(\beta_V(v)))\alpha_V(u))\Big)\beta_V(w)+l\Big(T(r(T(\beta_V(u)))\alpha_V(v))\Big)\beta_V(w)\\
&  = l(T(\alpha_V\beta_V(u)))l(T(\alpha_V(v)))w -(-1)^{|u||v|}l(T(\alpha_V\beta_V(v)))l(T(\alpha_V(u)))w\\
&  - l\Big(T(\beta_V(u))\circ T(\alpha_V(v))\Big)\beta_V(w)+(-1)^{|u||v|}l\Big(T(\beta_V(v))\circ T(\alpha_V(u))\Big)\beta_V(w)\\
&  \  = 0,
\end{align*}
it's since the equation \eqref{rep2}.\\
By the same way we show that
\begin{align*}
& \alpha\beta(x)\rhd(\alpha(y)\lhd z)-(\beta(x)\rhd \alpha(y))\lhd \beta(z)-(-1)^{|x||y|}\alpha\beta(y)\lhd (\alpha(x)\rhd z)\\
& -(-1)^{|x||y|}\alpha\beta(y)\lhd (\alpha(x) \lhd z)+(-1)^{|x||y|}(\beta(y)\lhd \alpha(x))\lhd \beta(z)\\
&   = 0.
\end{align*}
Therefore, $(V,\rhd,\lhd,\alpha_V,\beta_V)$ is a BiHom-$L$-dendriform superalgebra. The other conditions follow easily.
\end{proof}
A direct consequence of Theorem \ref{preLie=L-dendriform} is the following construction of a BiHom-$L$-dendriform superalgebra from a Rota-Baxter operator $($of weight zero$)$ of a BiHom-pre-Lie superalgebra.
\begin{cor}\label{RB-pre==L-den}Let $(\mathcal{A},\circ,\alpha,\beta)$ be a BiHom-pre-Lie superalgebra and $R$ be a Rota-Baxter operator on $\mathcal{A}$ $($of weight zero$)$. Assume that $R$ commute with $\alpha$ and $\beta$. We define the even binary operations $"\rhd"$ and $"\lhd"$ on $\mathcal{A}$ by
\begin{equation}\label{RB-pre=Ldend}
 x\rhd y=R(x)\circ y,~~~~x\lhd y=-(-1)^{|x||y|}\alpha^{-1}\beta(y)\circ R(\alpha\beta^{-1}(x)).
 \end{equation}
Then $(\mathcal{A},\rhd,\lhd,\alpha,\beta)$ is a BiHom-$L$-dendriform superalgebra.
\end{cor}
\begin{lem}Let $\{R_1,R_2\}$ be a pair of commuting Rota-Baxter operators $($of weight zero$)$ on a BiHom-pre-Lie superalgebra $(\mathcal{A},\circ,\alpha,\beta)$. Then $R_2$ is a
Rota-Baxter operator $($of weight zero$)$ on the BiHom-$L$-dendriform superalgebra $(\mathcal{A},\rhd,\lhd,\alpha,\beta)$ defined in \eqref{RB-pre=Ldend} with $R=R_1$.
\end{lem}
\begin{thm}Let $(\mathcal{A},\circ,\alpha,\beta)$ be a BiHom-pre-Lie superalgebra. Then there exists a compatible BiHom-$L$-dendriform superalgebra structure on $(\mathcal{A},\circ,\alpha,\beta)$ such that $(\mathcal{A},\circ,\alpha,\beta)$ is the associated vertical BiHom-pre-Lie superalgebra if and only if there exists an invertible super $\mathcal{O}$-operator $($of weight zero$)$ of $(\mathcal{A},\circ,\alpha,\beta)$.
\end{thm}
\begin{proof}Straightforward.
\end{proof}

\section{Cohomologies of BiHom-pre-Lie Superalgebras}

The cohomology theory of pre-Lie algebras is introduced in \cite{A-Dzhumadil-daev}, and in \cite{S-Liu-Songand-R-Tang} we find the extension in the Hom case . in \cite{Ben-Hassine-Chtioui-Mabrouk-Ncib} the authors introduced the comology theory of BiHom-pre-Lie algebras. in this section we extend this work in the super case.\\

Let $(V,l,r,\alpha_V,\beta_V)$ be a representation of a BiHom-pre-Lie superalgebra $(\mathcal{A},\cdot,\alpha,\beta)$.
An $n$-cochains  is a linear map $f:\wedge^{n-1} \mathcal{A}\otimes \mathcal{A}\rightarrow V$ such that it is compatible with $\alpha$, $\beta$ and $\alpha_V$, $\beta_V$ in the sense that $\alpha_V  f = f \alpha^{\otimes n},
\beta_V  f = f  \beta^{\otimes n}$.
The set of $n$-cochains is denoted by $ C^{n}(\mathcal{A};V),\ \  \forall n\geq 0$.

\

For all $f \in C^n(\mathcal{A};V),~ X=(x_1,\dots,x_{n+1}) \in \mathcal{H}(\mathcal{A})^{\otimes n}$, define the operator
$\partial^n:C^n(\mathcal{A};V)\longrightarrow C^{n+1}(\mathcal{A};V)$ by
\begin{eqnarray}\label{eq:12.1}
&&(\partial^n f)(x_1,\dots,x_{n+1})\\
 \nonumber&=&\sum_{i=1}^n(-1)^{i+1}(-1)^{|x_i||X|_{i-1}}l(\alpha^{n-1}\beta^{n-1}(x_i))f(\alpha(x_1),\dots,\widehat{\alpha(x_i)},\dots,\alpha(x_{n}),x_{n+1})\\
\nonumber &+&\sum_{i=1}^n(-1)^{i+1}-(1)^{|x_i||X|^{i+1}_n+|x_{n+1}||X|_n}r(\beta^{n-1}(x_{n+1}))f(\beta(x_1),\dots,\widehat{\beta(x_i)},\dots,\beta(x_n),\alpha^{n-1}(x_i))\\
\nonumber &-&\sum_{i=1}^n(-1)^{i+1}(-1)^{|x_i||X|^{i+1}_n}  f(\alpha\beta(x_1),\dots,\widehat{\alpha\beta(x_i)}\dots,\alpha\beta(x_n),\alpha^{n-1}(x_i)\cdot x_{n+1})\\
\nonumber&+&\sum_{1\leq i<j\leq n}(-1)^{i+j} (-1)^{\gamma_{ij}} f([\beta(x_i),\alpha(x_j)]_C,\alpha\beta(x_1),\dots,\widehat{\alpha\beta(x_i)},\dots,\widehat{\alpha\beta(x_j)},\dots,\alpha\beta(x_{n}),\beta(x_{n+1})),
\end{eqnarray}
where $|X|_k=|x_1|+\dots+|x_k|\;;\;|X|^l=|x_l|+\dots+|x_{n+1}|\;;\;\;|X|_k^l=|x_l|+\dots+|x_k|$ and $\gamma_{ij}=|x_i||X|_{i-1}+|x_j||X|_{j-1}+|x_i||x_j|$.

\begin{lem}
With the above notations, for any $f\in C^n(\mathcal{A};V)$, we have
\begin{eqnarray*}
&& \alpha_V (\partial^n f)=(\partial^n f) \alpha^{\otimes n+1}, \\
&&\beta_V (\partial^n f)=(\partial^n f) \beta^{\otimes n+1}.
\end{eqnarray*}
Thus we obtain a well-defined map
$$\partial^n:C^n(\mathcal{A};V)\longrightarrow C^{n+1}(\mathcal{A};V).$$
\end{lem}
\begin{proof}
Let $f \in C^n(\mathcal{A};V),~ X=(x_1,\dots,x_{n+1}) \in \mathcal{H}(\mathcal{A})^{\otimes n+1}$, we have
\begin{eqnarray*}
 &&(\partial^n f)(\alpha(x_1),\dots,\alpha(x_{n+1}))\\
 \nonumber&=&\sum_{i=1}^n(-1)^{i+1}(-1)^{|x_i||X|_{i-1}}l(\alpha^{n}\beta^{n-1}(x_i))
 f(\alpha^2(x_1),\dots,\widehat{\alpha^2(x_i)},\dots,\alpha^2(x_{n}),\alpha(x_{n+1}))\\
\nonumber &+&\sum_{i=1}^n(-1)^{i+1}(-1)^{|x_{n+1}||X|_n}(-1)^{|x_i||X|^{i+1}_{n}}r(\alpha\beta^{n-1}(x_{n+1}))
f(\alpha\beta(x_1),\dots,\widehat{\alpha\beta(x_i)},\dots,\alpha\beta(x_n),\alpha^n(x_i))\\
\nonumber &-&\sum_{i=1}^n(-1)^{i+1}(-1)^{|x_i||X|^{i+1}_n}
f(\alpha^2\beta(x_1),\dots,\widehat{\alpha^2\beta(x_i)}\dots,\alpha^2\beta(x_n),\alpha^n(x_i)\cdot x_{n+1}))\\
\nonumber&+&\sum_{1\leq i<j\leq n}(-1)^{i+j} (-1)^{\gamma_{ij}} f([\alpha\beta(x_i),\alpha^2(x_j)]_C,\alpha^2\beta(x_1),\dots,\widehat{\alpha^2\beta(x_i)},\dots,\widehat{\alpha^2\beta(x_j)},
\dots,\alpha^2\beta(x_{n}),\alpha\beta(x_{n+1}))\\
\nonumber&=&\sum_{i=1}^n(-1)^{i+1}(-1)^{|x_i||X|_{i-1}}\alpha_V l(\alpha^{n-1}\beta^{n-1}(x_i))f(\alpha(x_1),\dots,\widehat{\alpha(x_i)},\dots,\alpha(x_{n}),x_{n+1})\\
\nonumber &&+\sum_{i=1}^n(-1)^{i+1}(-1)^{|x_{n+1}||X|_n}(-1)^{|x_i||X|^{i+1}_{n}}\alpha_V r(\beta^{n-1}(x_{n+1}))f(\beta(x_1),\dots,\widehat{\beta(x_i)},\dots,\beta(x_n),\alpha^n(x_i))\\
\nonumber &&-\sum_{i=1}^n(-1)^{i+1}(-1)^{|x_i||X|^{i+1}_n} \alpha_V
f(\alpha\beta(x_1),\dots,\widehat{\alpha\beta(x_i)}\dots,\alpha\beta(x_n),\alpha^n(x_i)\cdot x_{n+1}))\\
\nonumber&&+\sum_{1\leq i<j\leq n}(-1)^{i+j}(-1)^{\gamma_{ij}} \alpha_V f([\beta(x_i),\alpha(x_j)]_C,\alpha\beta(x_1),\dots,\widehat{\alpha\beta(x_i)},\dots,\widehat{\alpha\beta(x_j)},\dots,\alpha\beta(x_{n}),\beta(x_{n+1}))\\
\nonumber&=&\alpha_V (\partial^n f)(x_1,\dots,x_{n+1}).
\end{eqnarray*}
Similarly, we obtain $(\partial^n f)(\beta(x_1),\dots,\beta(x_{n+1}))=\beta_V (\partial^n f)(x_1,\dots,x_{n+1}).$
\end{proof}
\begin{thm}\label{thm:operator}
The operator $\partial^n$ defined as above satisfies $\partial^{n+1}\circ\partial^{n}=0$.
\end{thm}
\begin{proof}
 Let $f \in C^n(\mathcal{A};V),~ X=(x_1,\dots,x_{n+1},x_{n+2}) \in \mathcal{H}(\mathcal{A})^{\otimes n+2}$, we have
\begin{eqnarray}
 \nonumber&&\partial^{n+1}\circ\partial^n (f)(x_1,\dots,x_{n+2})\\
 \nonumber&=&\sum_{i=1}^{n+1}(-1)^{i+1}(-1)^{|x_i||X|_{i-1}}l(\alpha^{n}\beta^{n}(x_i))
 \partial^nf(\alpha(x_1),\dots,\widehat{\alpha(x_i)},\dots,\alpha(x_{n+1},x_{n+2}))\\
\nonumber &+&\sum_{i=1}^{n+1}(-1)^{i+1}(-1)^{|x_{n+2}||X|_{n+1}+|x_i||X|^{i+1}_{n+1}}r(\beta^n(x_{n+2}))
\partial^nf(\beta(x_1),\dots,\widehat{\beta(x_i)},\dots,\beta(x_{n+1}),\alpha^n(x_i))\\
\nonumber &-&\sum_{i=1}^{n+1}(-1)^{i+1}  (-1)^{|x_i||X|^{i+1}_{n+1}}
\partial^nf(\alpha\beta(x_1),\dots,\widehat{\alpha\beta(x_i)}\dots,\alpha\beta(x_{n+1}),\alpha^n(x_i)\cdot x_{n+2})\\
\nonumber&+&\sum_{1\leq i<j\leq n+1}(-1)^{i+j} (-1)^{\gamma_{ij}} \partial^nf([\beta(x_i),\alpha(x_j)]_C,\alpha\beta(x_1),\dots,\widehat{\alpha\beta(x_i)},\dots,\widehat{\alpha\beta(x_j)}
,\dots,\alpha\beta(x_{n+1}),\beta(x_{n+2}))\\
\nonumber&=&\Theta_1+\Theta_2-\Theta_3+\Theta_4.
\end{eqnarray}

Where
\small{\begin{eqnarray}
 \nonumber\Theta_1
&=&\displaystyle\sum_{i=1}^{n+1}(-1)^{i+1}l(\alpha^{n}\beta^{n}(x_i))(-1)^{|x_i||X|_{i-1}}
\partial^nf(\alpha(x_1),\dots,\widehat{\alpha(x_i)},\dots,\alpha(x_{n+1}),x_{n+2})\\
 \label{e11}&=&\displaystyle\sum_{i=1}^{n+1}\displaystyle\sum_{\stackrel{j=1}{j<i}}^n(-1)^{i+j}(-1)^{|x_i||X|_{i-1}+|x_j||X|_{j-1}}
 l(\alpha^{n}\beta^{n}(x_i))ã(\alpha^{n}\beta^{n-1}(x_j))f(\alpha^2(x_1),..,\widehat{x_{j,i}}
,..,\alpha^2(x_{n+1}),x_{n+2})\\
\label{e12} &+&\displaystyle\sum_{i=1}^{n+1}\displaystyle\sum_{\stackrel{j=1}{j>i}}^n(-1)^{i+j-1}(-1)^{\gamma_{ij}}
l(\alpha^{n}\beta^{n}(x_i))l(\alpha^{n}\beta^{n-1}(x_j))
f(\alpha^2(x_1),..,\widehat{x_{i,j}},..,\alpha^2(x_{n+1}),x_{n+2})\\
\label{e13} &+&\displaystyle\sum_{i=1}^{n+1}\displaystyle\sum_{\stackrel{j=1}{j<i}}^n(-1)^{i+j}(-1)^{\theta_{ij}} l(\alpha^{n}\beta^{n}(x_i))r(\beta^{n-1}(x_{n+2}))
f(\alpha\beta(x_1),..,\widehat{x_{j,i}},..,\alpha\beta(x_{n+1}),
\alpha^n(x_j))\\
\label{e14} &+&\displaystyle\sum_{i=1}^{n+1}\displaystyle\sum_{\stackrel{j=1}{j>i}}^n(-1)^{i+j-1}(-1)^{\theta_{ij}+|x_i||x_j|}
l(\alpha^{n}\beta^{n}(x_i))r(\beta^{n-1}(x_{n+2}))
f(\alpha(x_1),..,\widehat{x_{i,j}},..,\alpha\beta(x_{n+1}),\alpha^n(x_{j}))\\
\label{e15} &-&\displaystyle\sum_{i=1}^{n+1}\displaystyle\sum_{\stackrel{j=1}{j<i}}^n(-1)^{i+j}(-1)^{\gamma_{ij}+|x_j||x_{n+2}|}
l(\alpha^{n}\beta^{n}(x_i))f(\alpha^2\beta(x_1),..,\widehat{x_{j,i}},..,\alpha^2\beta(x_{n+1}),\alpha^n(x_j)\cdot x_{n+2})\\
\label{e16} &-&\displaystyle\sum_{i=1}^{n+1}\displaystyle\sum_{\stackrel{j=1}{j>i}}^n(-1)^{i+j-1}(-1)^{|x_i||X|_{i-1}+|x_j||X|_{n+1}^{j+1}}
l(\alpha^{n}\beta^{n}(x_i))f(\alpha^2\beta(x_1),..,\widehat{x_{i,j}},..,\alpha^2\beta(x_{n+1}),\alpha^n(x_j)\cdot x_{n+2})\\
\label{e17} &-&\displaystyle\sum_{i=1}^{n+1}\displaystyle\sum_{{j<k<i}}(-1)^{i+j+k}(-1)^{\gamma_{jki}}
l(\alpha^{n}\beta^{n}(x_i))f([\alpha\beta(x_j),\alpha^2(x_k)]_C,\alpha^2\beta(x_1),..,\widehat{x_{j,k,i}},..,\alpha^2\beta(x_{n+1}),\beta(x_{n+2}))\\
\label{e18} &+&\displaystyle\sum_{i=1}^{n+1}\displaystyle\sum_{{j<i<k}}(-1)^{i+j+k}(-1)^{\gamma_{jik}}
l(\alpha^{n}\beta^{n}(x_i))f([\alpha\beta(x_j),\alpha^2(x_k)]_C,\alpha^2\beta(x_1),..,\widehat{x_{j,i,k}},..,\alpha^2\beta(x_{n+1}),\beta(x_{n+2}))\\
\label{e19} &-&\displaystyle\sum_{i=1}^{n+1}\displaystyle\sum_{{i<j<k}}(-1)^{i+j+k}(-1)^{\gamma_{ijk}}
l(\alpha^{n}\beta^{n}(x_i))f([\alpha\beta(x_j),\alpha^2(x_k)]_C,\alpha^2\beta(x_1),..,\widehat{x_{i,j,k}},..,\alpha^2\beta(x_{n+1}),\beta(x_{n+2}))
\end{eqnarray}}
where
\begin{eqnarray*}
\theta_{i<j}&=&|x_{n+2}||X|_{n+1}+|x_i||X|_{i-1}+|x_j||X|^{j+1}+|x_{n+2}|(|x_i|+|x_j|)\\
\gamma_{ijk}&=&|x_i||X|_{i-1}+|x_j||X|_{j-1}+|x_k||X|_{k-1}+|x_k|(|x_i|+|x_j|)+|x_i||x_j|.
\end{eqnarray*}
and
{\small\begin{align}
& \Theta_2\nonumber=\sum_{i=1}^{n+1}(-1)^{i+1}(-1)^{|x_{n+2}||X|_{n+1}+|x_i||X|^{i+1}_{n+1}}
r(\beta^n(x_{n+2}))\partial^nf(\beta(x_1),..,\widehat{\beta(x_i)},..,\beta(x_{n+1}),\alpha^n(x_i))\\
 \label{e21}&=\sum_{i=1}^{n+1}\sum_{\stackrel{j=1}{j<i}}^n(-1)^{i+j}(-1)^{\theta_{ji}+|x_j||x_{n+2}|}r(\beta^n(x_{n+2}))l(\alpha^{n-1}\beta^{n}(x_j))
 f(\alpha\beta(x_1),..,\widehat{x_{j,i}},..,\alpha\beta(x_{n+1}),\alpha^n(x_i))\\
 \label{e22}&+\sum_{i=1}^{n+1}\sum_{\stackrel{j=1}{j>i}}^n(-1)^{i+j-1}(-1)^{\theta_{ji}+|x_j|(|x_i|+|x_{n+2}|)}
r(\beta^n(x_{n+2}))l(\alpha^{n-1}\alpha^{n-1}\beta^{n}(x_j))
 f(\alpha\beta(x_1),..,\widehat{x_{i,j}},..,\alpha\beta(x_{n+1}),\alpha^n(x_i))\\
 \label{e23}&+\sum_{i=1}^{n+1}\sum_{\stackrel{j=1}{j<i}}^n(-1)^{i+j}(-1)^{\theta_{ij}+|x_i|(|x_j|+|x_{n+2}|)}r(\beta^n(x_{n+2}))r(\alpha^n\beta^{n-1}(x_i))
 f(\beta^2(x_1),..,\widehat{x_{j,i}},..,\alpha^{n-1}(x_j))\\
 \label{e24}&+\sum_{i=1}^{n+1}\sum_{\stackrel{j=1}{j>i}}^n(-1)^{i+j-1}(-1)^{\theta_{ij}+|x_i||x_{n+2}|}r(\beta^n(x_{n+2}))r(\alpha^n\beta^{n-1}(x_i))
 f(\beta^2(x_1),..,\widehat{x_{i,j}},..,\alpha^{n-1}\beta(x_j))\\
  \label{e25}&-\sum_{i=1}^{n+1}\sum_{\stackrel{j=1}{j<i}}^n(-1)^{i+j}(-1)^{|x_i||X|^{i+1}+|x_j||X|^{j+1}+|x_{n+2}||X|_{n+1}}r(\beta^n(x_{n+2})) f(\alpha\beta^2(x_1),..,\widehat{x_{j,i}},..,\alpha\beta^2(x_{n+1}),\alpha^{n-1}\beta(x_j)\cdot\alpha^n(x_i))\\
\label{e26}&-\sum_{i=1}^{n+1}\sum_{\stackrel{j=1}{j>i}}^n(-1)^{i+j-1}(-1)^{|x_i||X|^{i+1}+|x_j||X|^{j+1}+|x_{n+2}||X|_{n+1}}(-1)^{|x_i||x_j|}r
(\beta^n(x_{n+2}))\\
\nonumber&f(\alpha\beta^2(x_1),..,\widehat{x_{i,j}},..,\alpha\beta^2(x_{n+1}),\alpha^{n-1}\beta(x_j)\cdot\alpha^n(x_i))\\
\label{e27} &-\sum_{i=1}^{n+1}\sum_{{j<k<i}}(-1)^{i+j+k}(-1)^{|x_i|(|X|^{i+1}+|x_{n+2}|)+|x_j|(|X|_{j-1}+|x_k|)}(-1)^{|x_k||X|_{k-1}+|x_{n+2}||X|_{n+1}}\\
\nonumber&r(\beta^n(x_{n+2}))
f([\beta^2(x_j),\alpha^2(x_k)]_C,\alpha\beta^2(x_1),..,\widehat{x_{j,k,i}},..,\alpha\beta^2(x_{n+1}),\alpha^n\beta(x_{i}))\\
\label{e28} &+\sum_{i=1}^{n+1}\sum_{{j<i<k}}(-1)^{i+j+k}(-1)^{|x_i|(|X|^{i+1}+|x_k|+|x_{n+2}|)}(-1)^{|x_j|(|X|_{j-1}+|x_k|)+|x_k||X|_{k-1}+|x_{n+2}||X|_{n+1}}\\
\nonumber&r(\beta^n(x_{n+2}))
f([\beta^2(x_j),\alpha^2(x_k)]_C,\alpha\beta^2(x_1),..,\widehat{x_{j,i,k}},..,\alpha\beta^2(x_{n+1}),\alpha^n\beta(x_{i}))\\
\label{e29} &-\sum_{i=1}^{n+1}\sum_{{i<j<k}}(-1)^{i+j+k}R(\beta^n(x_{n+2}))(-1)^{|x_i|(|X|^{i+1}+|x_k|+|x_j|+|x_{n+2}|)}
(-1)^{|x_j|(|X|_{j-1}+|x_k|)+|x_k||X|_{k-1}+|x_{n+2}||X|_{n+1}}\\
\nonumber&f([\beta^2(x_j),\alpha^2(x_k)]_C,\alpha\beta^2(x_1),..,\widehat{x_{i,j,k}},..,\alpha\beta^2(x_{n+1}),\alpha^n\beta(x_{i})),
\end{align}}
and
{\small\begin{align}
\nonumber&\Theta_3=\sum_{i=1}^{n+1}(-1)^{i+1}(-1)^{|x_i||X|^{i+1}_{n+1}}  \partial^nf(\alpha\beta(x_1),..,\widehat{\alpha\beta(x_i)}..,\alpha\beta(x_{n+1}),\alpha^n(x_i)\cdot x_{n+2})\\
\label{e31}&=\sum_{i=1}^{n+1}\sum_{\stackrel{j=1}{j<i}}^n(-1)^{i+j}(-1)^{|x_j||X|_{j-1}+|x_i||X|^{i+1}_{n+1}}
l(\alpha^n\beta^n(x_{j})) f(\alpha^2\beta(x_1),..,\widehat{x_{j,i}},..,\alpha^2\beta(x_{n+1}),\alpha^n(x_i)\cdot x_{n+2})\\
  \label{e32}&+\sum_{i=1}^{n+1}\sum_{\stackrel{j=1}{j>i}}^n(-1)^{i+j-1}(-1)^{|x_j|(|X|_{j-1}+|x_i|)+|x_i||X|^{i+1}_{n+1}}
l(\alpha^n\beta^n(x_{j})) f(\alpha^2\beta(x_1),..,\widehat{x_{i,j}},..,\alpha^2\beta(x_{n+1}),\alpha^n(x_i)\cdot x_{n+2})\\
 \label{e33}&+\sum_{i=1}^{n+1}\sum_{\stackrel{j=1}{j<i}}^n(-1)^{i+j}(-1)^{\theta_{ij}+|x_j|(|x_i|+|x_{n+2}|)}
r(\alpha^n\beta^{n-1}(x_i)\cdot\beta^{n-1}(x_{n+2}))
 f(\alpha\beta^2(x_1),..,\widehat{x_{j,i}},..,\alpha^2\beta(x_{n+1}),\alpha^n\beta(x_j))\\
 \label{e34}&+\sum_{i=1}^{n+1}\sum_{\stackrel{j=1}{j>i}}^n(-1)^{i+j}(-1)^{\theta_{ij}+|x_j||x_{n+2}|}
r(\alpha^n\beta^{n-1}(x_i)\cdot\beta^{n-1}(x_{n+2}))
 f(\alpha\beta^2(x_1),..,\widehat{x_{i,j}},..,\alpha^2\beta(x_{n+1}),\alpha^n\beta(x_j))\\
\label{e35}&-\sum_{i=1}^{n+1}\sum_{\stackrel{j=1}{j<i}}^n(-1)^{i+j} (-1)^{|X|^{i+1}_{n+1}(|x_i|+|x_j|)+|x_j||X|^{j+1}_{i-1}} f(\alpha^2\beta^2(x_1),..,\widehat{x_{j,i}},..,\alpha^2\beta^2(x_{n+1}),(\alpha^n\beta(x_{j}))\cdot(\alpha^n(x_i)\cdot x_{n+2}))\\
 \label{e36}&-\sum_{i=1}^{n+1}\sum_{\stackrel{j=1}{j>i}}^n(-1)^{i+j} (-1)^{|X|^{j+1}_{n+1}(|x_i|+|x_j|)+|x_i||X|^{j}_{i-1}} f(\alpha^2\beta^2(x_1),..,\widehat{x_{i,j}},..,\alpha^2\beta^2(x_{n+1}),(\alpha^n\beta(x_{j}))\cdot(\alpha^n(x_i)\cdot x_{n+2}))\\
 \label{e37} &-\sum_{i=1}^{n+1}\sum_{{j<k<i}}(-1)^{i+j+k}(-1)^{|x_j|(|X|_{j-1}+|x_k|)+|x_k||X|_{k-1}+|x_i||X|_{n+1}^{i+1}}\nonumber\\
 &f([\alpha\beta^2(x_j),\alpha^2\beta(x_k)]_C,\alpha^2\beta^2(x_1),..,\widehat{x_{j,k,i}},..,\alpha^2\beta^2(x_{n+1}),\alpha^n\beta(x_i)\cdot \beta(x_{n+2}))\\
\label{e38} &+\sum_{i=1}^{n+1}\sum_{{j<i<k}}(-1)^{i+j+k}(-1)^{|x_j|(|X|_{j-1}+|x_k|)+|x_k|(|X|_{k-1}+|x_i|)+|x_i||X|_{n+1}^{i+1}}\nonumber\\
&f([\alpha\beta^2(x_j),\alpha^2\beta(x_k)]_C,\alpha^2\beta^2(x_1),..,\widehat{x_{j,i,k}},..,\alpha^2\beta^2(x_{n+1}),\alpha^n\beta(x_i)\cdot\beta(x_{n+2}))\\
\label{e39} &-\sum_{i=1}^{n+1}\sum_{{i<j<k}}(-1)^{i+j+k}(-1)^{|x_j|(|X|_{j-1}+|x_i|+|x_k|)+|x_k|(|X|_{k-1}+|x_i|)+|x_i||X|_{n+1}^{i+1}}\nonumber\\
&f([\alpha\beta^2(x_j),\alpha^2\beta(x_k)]_C,\alpha^2\beta^2(x_1),..,\widehat{x_{i,j,k}},..,\alpha^2\beta^2(x_{n+1}),\alpha^n\beta(x_i)\cdot \beta(x_{n+2}))
\end{align}}
and
{\small\begin{align}
\nonumber&\Theta_4= \sum_{1\leq i<j\leq n+1}(-1)^{i+j} (-1)^{\gamma_{ij}} \partial^nf([\beta(x_i),\alpha(x_j)]_C,\alpha\beta(x_1),..,\widehat{x_{i,j}},..,\alpha\beta(x_{n+1}),\beta(x_{n+2}))\\
\label{e41}&=\sum_{\stackrel{1\leq i< j\leq n+1}{}}(-1)^{i+j}(-1)^{\gamma_{ij}}l([\alpha^{n-1}\beta^n(x_i),\alpha^n\beta^{n-1}(x_j)]_C) f(\alpha^2\beta(x_1),..,\widehat{x_{i,j}},..,\alpha^2\beta(x_{n+1}),\beta(x_{n+2}))\\
\label{e42}&+\sum_{\stackrel{1\leq i< j}{}}^{ n+1}\sum_{{k<i}}(-1)^{i+j+k+1}(-1)^{\gamma_{ij}+|x_k|(|X|_{k-1}+|x_i|+|x_j|)}
l(\alpha^n\beta^n(x_k)) f([\alpha\beta(x_i),\alpha^2(x_j)]_C ,\alpha^2\beta(x_1),..,\widehat{x_{k,i,j}},..\alpha^2\beta(x_{n+1}),\beta(x_{n+2}))\\
\label{e43}&+\sum_{\stackrel{1\leq i< j}{}}^{ n+1}\sum_{{i<k<j}}(-1)^{i+j+k}(-1)^{\gamma_{ij}+|x_k|(|X|_{k-1}+|x_j|)}
l(\alpha^n\beta^n(x_k))f([\alpha\beta(x_i),\alpha^2(x_j)]_C ,\alpha^2\beta(x_1),..,\widehat{x_{i,k,j}},..\alpha^2\beta(x_{n+1}),\beta(x_{n+2}))\\
\label{e44}&+\sum_{\stackrel{1\leq i< j}{}}^{ n+1}\sum_{{j<k}}(-1)^{i+j+k+1}(-1)^{\gamma_{ij}+|x_k||X|_{k-1}}
l(\alpha^n\beta^n(x_k)) f([\alpha\beta(x_i),\alpha^2(x_j)]_C ,\alpha^2\beta(x_1),..,\widehat{x_{i,j,k}},..\alpha^2\beta(x_{n+1}),\beta(x_{n+2}))\\
\label{e45}&+\sum_{\stackrel{1\leq i< j}{}}^{ n+1}(-1)^{i+j} (-1)^{|x_i|(|X|^{i+1}_{n+1}+|x_j|)+|x_j||X|^{j+1}_{n+1}+|x_{n+2}||X|_{n+1}}
r(\beta^n(x_{n+2})) f(\alpha\beta^2(x_1),..,\widehat{x_{i,j}},..,\alpha\beta^2(x_{n+1}),[\alpha^{n-1}\beta(x_i),\alpha^n(x_j)]_C)\\
\label{e46}&-\sum_{\stackrel{1\leq i< j}{}}^{n+1}\sum_{{k<i}}(-1)^{i+j+k}(-1)^{\gamma_{ij}+|x_k|(|X|^{k+1}_{n+1}+|x_i|+|x_j|)+|x_{n+2}||X|_{n+1}}
r(\beta^n(x_{n+2})) f([\beta^2(x_i),\alpha\beta(x_j)]_C,\alpha\beta^2(x_1),..,\widehat{x_{k,i,j}},..,\alpha\beta^2(x_{n+1}),\alpha^n\beta(x_k))\\
\label{e47}&+\sum_{\stackrel{1\leq i< j}{}}^{ n+1}\sum_{{i<k<j}}(-1)^{i+j+k}(-1)^{\gamma_{ij}+|x_k|(|X|^{k+1}_{n+1}+|x_j|)+|x_{n+2}||X|_{n+1}}
r(\beta^n(x_{n+2})) f([\beta^2(x_i),\alpha\beta(x_j)]_C,\alpha\beta^2(x_1),..,\widehat{x_{i,k,j}},..,\alpha\beta^2(x_{n+1}),\alpha^n\beta(x_k))\\
\label{e48}&-\sum_{\stackrel{1\leq i<j}{}}^{n+1}\sum_{{j<k}}(-1)^{i+j+k}(-1)^{\gamma_{ij}+|x_k||X|^{k+1}_{n+1}+|x_{n+2}||X|_{n+1}}
r(\beta^n(x_{n+2})) f([\beta^2(x_i),\alpha\beta(x_j)]_C,\alpha\beta^2(x_1),..,\widehat{x_{i,j,k}},..,\alpha\beta^2(x_{n+1}),\alpha^n\beta(x_k))
\end{align}
\begin{align}
\label{e51}-&\sum_{\stackrel{1\leq i< j}{}}^{n+1}(-1)^{i+j}(-1)^{|x_i|(|X|_{n+1}^{i+1}+|x_j|)+|x_j||X|_{n+1}^{j+1}}  f(\alpha^2\beta^2(x_1),..,\widehat{x_{i,j}},..,\alpha^2\beta^2(x_{n+1}),[\alpha^{n-1}\beta(x_i),\alpha^n(x_j)]_C\cdot \beta^n(x_{n+2}))\\
\label{e52}+&\sum_{\stackrel{1\leq i< j}{}}^{n+1}\sum_{{k<i}}(-1)^{i+j+k}(-1)^{\gamma_{ij}+|x_k|(|X|_{n+1}^{k+1}+|x_i|+|x_j|)} f([\alpha\beta^2(x_i),\alpha^2\beta(x_j)],\alpha^2\beta^2(x_1),..,\widehat{x_{k,i,j}},..,\alpha^2\beta^2(x_{n+1}),\alpha^n\beta(x_k)\cdot \beta(x_{n+2}))\\
\label{e53}-&\sum_{\stackrel{1\leq i< j}{}}^{n+1}\sum_{{i<k<j}}(-1)^{i+j+k} (-1)^{\gamma_{ij}+|x_k|(|X|_{n+1}^{k+1}+|x_j|)} f([\alpha\beta^2(x_i),\alpha^2\beta(x_j)],\alpha^2\beta^2(x_1),..,\widehat{x_{i,k,j}},..,\alpha^2\beta^2(x_{n+1}),\alpha^n\beta(x_k)\cdot \beta(x_{n+2}))\\
\label{e54}+&\sum_{\stackrel{1\leq i< j}{}}^{n+1}\sum_{{j<k}}(-1)^{i+j+k}(-1)^{\gamma_{ij}+|x_k||X|_{n+1}^{k+1}}
f([\alpha\beta^2(x_i),\alpha^2\beta(x_j)],\alpha^2\beta^2(x_1),..,\widehat{x_{i,j,k}},..,\alpha^2\beta^2(x_{n+1}),\alpha^n\beta(x_k)\cdot \beta(x_{n+2}))\\
\label{e55}-&\sum_{\stackrel{1\leq i< j}{}}^{n+1}\sum_{{l<i}}(-1)^{i+j+l}(-1)^{\gamma_{lij}}
 f([[\beta^2(x_i),\alpha\beta(x_j)],\alpha^2\beta(x_l)],\alpha^2\beta^2(x_1),..,\widehat{x_{l,i,j}},..,
\alpha^2\beta^2(x_{n+1}),\beta^2(x_{n+2}))\\
\label{e56}+&\sum_{\stackrel{1\leq i< j}{}}^{n+1}\sum_{{i<l<j}}(-1)^{i+j+l}(-1)^{\gamma_{ilj}}
 f([[\beta^2(x_i),\alpha\beta(x_j)],\alpha^2\beta(x_l)],\alpha^2\beta^2(x_1),..,\widehat{x_{i,l,j}},..,
\alpha^2\beta^2(x_{n+1}),\beta^2(x_{n+2}))\\
\label{e57}-&\sum_{\stackrel{1\leq i< j}{}}^{n+1}\sum_{{j<l}}(-1)^{i+j+l}(-1)^{\gamma_{ijl}}
f([[\beta^2(x_i),\alpha\beta(x_j)],\alpha^2\beta(x_l)],\alpha^2\beta^2(x_1),..,\widehat{x_{i,j,l}},..,
\alpha^2\beta^2(x_{n+1}),\beta^2(x_{n+2}))\\
\nonumber&+\sum_{\stackrel{1\leq i< j}{}}^{n+1}\sum_{{k<l<i}}(-1)^{i+j+l+k}(-1)^{\gamma_{ij}+\gamma_{kl}+(|x_i|+|x_j|)(|x_k|+|x_l|)}
f([\alpha\beta^2(x_k),\alpha^2\beta(x_l)]_C,[\alpha\beta^2(x_i),\alpha^2(x_j)]_C,\alpha^2\beta^2(x_1),..,\\
\label{e58}&\widehat{x_{k,l,i,j}},..,
\alpha^2\beta^2(x_{n+1}),\beta^2(x_{n+2}))\\
\nonumber&-\sum_{\stackrel{1\leq i< j}{}}^{n+1}\sum_{{k<i<l<j}}(-1)^{i+j+l+k}(-1)^{\gamma_{ij}+\gamma_{kl}+|x_k|(|x_i|+|x_j|)+|x_l||x_j|}
 f([\alpha\beta^2(x_k),\alpha^2\beta(x_l)]_C,[\alpha\beta^2(x_i),\alpha^2(x_j)]_C,\alpha^2\beta^2(x_1),..,\\
\label{e59}&\widehat{x_{k,i,l,j}},..,
\alpha^2\beta^2(x_{n+1}),\beta^2(x_{n+2}))\\
\nonumber&+\sum_{\stackrel{1\leq i< j}{}}^{n+1}\sum_{{k<i<j<l}}(-1)^{i+j+l+k}(-1)^{\gamma_{ij}+\gamma_{kl}+|x_k|(|x_i|+|x_j|)}
f([\alpha\beta^2(x_k),\alpha^2\beta(x_l)]_C,[\alpha\beta^2(x_i),\alpha^2(x_j)]_C,\alpha^2\beta^2(x_1),..,\\
\label{e510}&\widehat{x_{k,i,j,l}},..,\alpha^2\beta^2(x_{n+1}),\beta^2(x_{n+2}))\\
\nonumber&-\sum_{\stackrel{1\leq i<j}{}}^{n+1}\sum_{{i<k<l<j}}(-1)^{i+j+l+k}(-1)^{\gamma_{ij}+\gamma_{kl}+|x_j|(|x_k|+|x_l|)}
f([\alpha\beta^2(x_k),\alpha^2\beta(x_l)]_C,[\alpha\beta^2(x_i),\alpha^2(x_j)]_C,\alpha^2\beta^2(x_1),..,\\
\label{e511}&\widehat{x_{i,k,l,j}},..,
\alpha^2\beta^2(x_{n+1}),\beta^2(x_{n+2}))\\
\nonumber&+\sum_{\stackrel{1\leq i< j}{}}^{n+1}\sum_{{i<k<j<l}}(-1)^{i+j+l+k}(-1)^{\gamma_{ij}+\gamma_{kl}+|x_j||x_k|}
 f([\alpha\beta^2(x_k),\alpha^2\beta(x_l)]_C,[\alpha\beta^2(x_i),\alpha^2(x_j)]_C,\alpha^2\beta^2(x_1),..,\\
\label{e512}&\widehat{x_{i,k,j,l}},..,
\alpha^2\beta^2(x_{n+1}),\beta^2(x_{n+2}))\\
\nonumber&-\sum_{\stackrel{1\leq i< j}{}}^{n+1}\sum_{{j<k<l}}(-1)^{i+j+l+k}(-1)^{\gamma_{ij}+\gamma_{kl}}
 f([\alpha\beta^2(x_k),\alpha^2\beta(x_l)]_C,[\alpha\beta^2(x_i),\alpha^2(x_j)]_C,\alpha^2\beta^2(x_1),..,\\
\label{e513}&\widehat{x_{i,j,k,l}},..,
\alpha^2\beta^2(x_{n+1}),\beta^2(x_{n+2})).
\end{align}}
In $\Theta_1,..., \Theta_4$,  $\widehat{x_{i, j, k, l}}$ means that we omits the items $x_i, x_j, x_k, x_l$.\\

We use the same idea used in \cite{Ben-Hassine-Chtioui-Mabrouk-Ncib}, where each element multiplied by their cofficients according to their parity, and by a direct computation we find the result.

%
\end{proof}

Denote the set of  $n$-cocycles by $Z^n(A;V)=\ker(\partial^n)$ and the set of  $n$-cobords by $B^n(A;V)=Im(\partial^{n-1})$. We denote by $H^n(A;V)=Z^n(A;V)/B^n(A;V)$ the corresponding cohomology groups of the BiHom-pre-Lie superalgebra $(A,\cdot,\alpha,\beta)$ with the coefficient in the representation $(V,l,r,\alpha_V,\beta_V)$.

\paragraph{\textbf{Acknowledgment.} We would like to  thank  Abdelkader Ben Hassine, Sami Mabrouk and Taoufik Chtioui for them valuable remarks and suggestions. }


\begin{thebibliography}{10}

\bibitem{Abdaoui-Mabrouk-Makhlouf}
El-Kadri Abdaoui, Sami Mabrouk and Abdenacer Makhlouf,
\newblock {\it Rota-Baxter operators on Pre-Lie Superalgebras},
\newblock Bull. Malays. Math. Sci. Soc
\newblock, vol. \textbf{42}: (2019), 1567-1606.
\bibitem{Aguiar2004}
M. Aguiar,
\newblock {\it Infinitesimal bialgebras, pre-Lie algebras and dendriform algebras, in "Hopf algebras"},
\newblock  Lecture Notes in Pure and Appl. Math
\newblock, vol. \textbf{237}: (2004), 1--33.

\bibitem{Aguiar2000}
M. Aguiar,
\newblock {\it Pre-Poisson algebras},
\newblock  Lett. Math. Phys
\newblock, vol. \textbf{54}: (2000), 263--277.

\bibitem{Aguiar-Loday2004}
M. Aguiar and J. L. Loday,
\newblock {\it Quadri-algebras},
\newblock  J. Pure Appl Algebra
\newblock, vol. \textbf{191}: (2004), 251--221.

\bibitem{Aizawa-N-Sato-H}
Aizawa. N and Sato. H,
\newblock{\it $q$-deformation of the Virasoro algebra with central extension},
\newblock Phys. Letter. B
\newblock, vol. \textbf{256}(1991), 185-190.


\bibitem{Ammar-Makhlouf}
F. Ammar and A. Makhlouf,
\newblock {\it Hom-Lie superalgebras and Hom-Lie admissible superalgebras},
\newblock J. Algebra
\newblock {\bf 324} (2010), 1513--1528.


\bibitem{Andrada-Salamon}
A. Andrada and S. Salamon,
\newblock {\it Complex product structure on Lie algebras},
\newblock Forum  Math
\newblock {\bf 17} (2005), 261--295.

\bibitem{Atkinson1967}
F. V. Atkinson,
\newblock {\it Some aspects of Baxter's functional equation},
\newblock J. Math. Anal. Appl
\newblock, vol. \textbf{7}: (1967), 1--30.

\bibitem{Baxter1960}
G. Baxter,
\newblock {\it An analytic problem whose solution follows from a simple algebraic identity},
\newblock  Pacific J. Math
\newblock, vol. \textbf{10}: (1960), 731-742.

\bibitem{Bai phase spaces}
 C. M. Bai,
\newblock {\it A further study on non-abelian phase spaces: Left-symmetric algebraic approach and related geometry },
\newblock Rev. Math. Phys,
\newblock vol. \textbf{18} (2006), 545--564


\bibitem{Bai O-operators}
 C. M. Bai,
\newblock {\it $\mathcal{O}$-operators of Loday algebras and analogues of the classical Yang-Baxter equation},
\newblock Comm. Algebra,
\newblock vol. \textbf{38} (2010), 4277--4321.

\bibitem{Bai-algebraic}
C. M. Bai,
\newblock {\it A unified algebraic approach to classical Yang-Baxter equation},
\newblock J. Phys. A: Math. Theor
\newblock {\bf 40} (2007), 11073--11082.





\bibitem{Bai-Liu OOperaorLie}
C. M. Bai, L. Guo, X. Ni,
\newblock {\it O-operators on associative algebras and associative Yang-Baxter equations},
\newblock Pacific Journal of Mathematics
\newblock {\bf 256} (2012), 257--289.

\bibitem{Bai-Liu OOperaorAss}
C. M. Bai, L. Guo, X. Ni,
\newblock {\it Generalizations of the classical  Yang-Baxter equation and O-operators},
\newblock Journal of Mathematical Physics
\newblock {\bf } (2011), 52:063515.

\bibitem{Bai-Liu L-dendrifom}
C. M. Bai, L. G. Liu,
\newblock {\it Some results on $L$-dendriform algebras},
\newblock J. Geom Phys
\newblock {\bf 60} (2010), 940--950

\bibitem{Bai and Zhang classif}
C. M. Bai and R. Zhang,
\newblock {\it On some left-symmetric superalgebras},
\newblock J. of Algebra and its applications,
\newblock vol. \textbf{11} (5) (2012).

\bibitem{IbrahimaBakayoko}
I. Bakayoko,
\newblock {\it Hom-post-Lie modules, O-operators and some functors on Hom-algebras},
\newblock arXiv preprint,
\newblock arXiv:1610.02845, 2016.

\bibitem{Ben-Hassine-Chtioui-Mabrouk-Ncib}
 Abdelkader Ben Hassine, Taoufik Chtioui, Sami Mabrouk, Othmen Ncib,
\newblock {\it Cohomology and linear deformation of BiHom-left-symmetric algebras},
\newblock arXiv preprint,
\newblock arXiv:1907.06979.

\bibitem{Bordemann}
M. Bordemann,
\newblock {\it Generalized Lax pairs, the modified classical Yang-Baxter equations, and affine geometry of Lie groups},
\newblock Comm Math. Phys,
\newblock vol. \textbf{135} (1) (1990), 201--216.




\bibitem{Burde}
D. Burde,
\newblock {\it Left-symmetric algebras, or pre-Lie algebras in geometry and physics},
\newblock Cent. Eur. J. Math,
\newblock vol. \textbf{4} (3) (2006), 323--357.

\bibitem{Cartier1972}
P. Cartier,
\newblock {\it On the structure of free Baxter algebras},
\newblock Advences in Math
\newblock {\bf 9} (1972), 253--265.

\bibitem{Chapoton-Livernet}
F. Chapoton, M. Livernet,
\newblock {\it Pre-Lie algebras and the rooted trees operad},
\newblock Internat. Math. Res.
\newblock Notices {\bf 8} (2001), 395--408.

%
%
%

\bibitem{Cheng-Qi}
 Y. Cheng and H. Qi,
\newblock {\it Representations of BiHom-Lie algebras},
\newblock arXiv preprint,
\newblock arXiv:1610.04302v1(2016).

\bibitem{Chaichian-Kulish-Lukierski}
Chaichian. M, Kulish. Pand Lukierski. J,
\newblock {\it $q$-deformed Jacobi identity, $q$-oscillators and $q$-deformed infinitedimensional algebras},
\newblock Phys. Lett. B,
\newblock vol. \textbf{237}(1990), 401-406.

\bibitem{Chu}
B. Y. Chu,
\newblock {\it Symplectic homogeneous spaces},
\newblock Trans. Amer. Math. Soc
\newblock {\bf  197} (1974), 145--159.

\bibitem{Connes-Kreimer}
A. Connes and D. Kreimer,
\newblock {\it Hopf algebras, renormalization and noncommutative geometry},
\newblock Comm. Math. Phys
\newblock {\bf  199} (1998), 203--242.

\bibitem{Curtright-Zachos}
Curtright T.L and Zachos C.K,
\newblock {\it Deforming maps for quantum algebras},
\newblock Phys. Lett. B
\newblock vol. \textbf{243} (1990), 237–244.


\bibitem{Dardié-Medina1}
J. M. Dardi\'{e} and A. M\'{e}dina,
\newblock {\it Alg\`{e}bres de Lie Kahl\'{e}riennes et double extension},
\newblock J. Algebra
\newblock {\bf  185} (1996), 744--795.

\bibitem{Dardié-Medina2}
J. M. Dardi\'{e} and A. M\'{e}dina,
\newblock {\it Double extension symplectique d'un groupe de Lie symplectique},
\newblock Adv. Math
\newblock {\bf  117} (1996), 208--227.

\bibitem{Diata-Medina}
A. Diatta and A. Medina,
\newblock {\it Classical Yang-Baxter equation and left-invariant affine geometry on Lie groups},
\newblock Manuscipta. Math
\newblock {\bf  114} (2004), 477--486.

\bibitem{A-Dzhumadil-daev}
A. Dzhumadil'daev,
\newblock {\it Cohomologies and deformations of right-symmetric algebras},
\newblock J. Math. Sci.
\newblock {\bf 93} (1999), no. 6, 836-876.

\bibitem{Ebrahimi-loday.algebras}
K. Ebrahimi-Fard,
\newblock {\it Loday-type algebras and the Rota-Baxter relation},
\newblock Lett. Math. Phys
\newblock {\bf 61} (2002), 139--147.

\bibitem{Ebrahimi-assoc-Nijenhuis}
K. Ebrahimi-Fard,
\newblock {\it On the associative Nijenhuis relation},
\newblock  Elect. J. Comb.
\newblock {\bf 11} no. 1, (2004).



\bibitem{Ebrahimi-Manchon-Patr}
K. Ebrahimi-Fard, D. Manchon, F. Patras,
\newblock {\it New identities in dendriform algebras},
\newblock J. Algebra
\newblock {\bf 320} (2008), 708--727.

%


\bibitem{Gerstenhaber}
M. Gerstenhaber,
\newblock {\it The cohomology structure of associative ring},
\newblock Ann. of Math
\newblock {\bf 78} (1963), 267--288.

\bibitem{Goze}
M. Goze and E. Remm,
\newblock {\it Lie-admissible algebras and operads},
\newblock J. Algebra
\newblock {\bf 273} (2004), 129--152.

\bibitem{Graziani-Makhlouf-Menini-Panaite}
G. Graziani, A. Makhlouf, C. Menini and F. Panaite
\newblock {\it BiHom-associative algebras, BiHom-Lie algebras and BiHom-bialgebras},
\newblock  Symmetry, Integrability and geometry, SIGMA
\newblock  \textbf{\textbf{11}} (2015), 086, 34 pages.



\bibitem{Guo-Introd}
L. Guo,
\newblock {\it An introduction to Rota-Baxter Algebra},
\newblock http://andromeda.ritgers.edu/Liguo/rbabook.pdf.


%

\bibitem{Koszul}
J. -L. Koszul,
\newblock {\it Domaines born\'{e}s homog\`{e}nes et orbites de groupes de transformations affines},
\newblock Bull. Soc. Math. France
\newblock {\bf  89} (1961), 515--533.


\bibitem{Kupershmidt2}
B. A. Kupershmidt,
\newblock {\it Non-abelian phase spaces},
\newblock J. Phys. A: Math. Gen,
\newblock vol. \textbf{27} (1994), 2801-2809.




\bibitem{Lichnerowicz-Medina}
A. Lichnerowicz and A. Medina,
\newblock {\it On Lie group with left-invariant symplectic or K\"{a}hlerian},
\newblock Lett. Math. Phys.
\newblock  {\bf 16} (3) (1988), 225--235.

\bibitem {S-Liu-Songand-R-Tang}
S. Liu, Songand R. Tang,
\newblock {\it Representations and cohomologies of Hom-pre-Lie algebras},
\newblock arXiv preprint,
\newblock arXiv:1902.07360 (2019).

\bibitem {Liu-K-Q}
Liu K.Q,
\newblock{\it Characterizations of the quantum Witt algebra},
\newblock Lett. Math. Phys.
\newblock vol \textbf{24} (1992), 257–265.

\bibitem{Loday-dialgebras}
J.-L. Loday,
\newblock {\it Dialgebras, in: Dialgebras and related Operads, in},
\newblock Lecture Notes in Math,
\newblock vol. \textbf{} (1763) Springer, New York. (2001), 7--66.



\bibitem{Miller1969}
J. B. Miller,
\newblock {\it Baxter operators and endomorphisms on Banach algebras},
\newblock J. Math. Anal. Appl,
\newblock {\bf 25} (1969), 503--520.


\bibitem{Rota1995}
G. -C. Rota,
\newblock {\it Baxter operators},
\newblock In Gian-Carlo Rota on Combinatorics, Introductory Papes and commentaries, edited by Joseph. P. S. Kung,
\newblock $($Birkhauser, Boston, 1995$)$.


\bibitem{WG-BiHom-Lie supealgebras}
Wang. S.,  Guo. S,
\newblock {\it BiHom-Lie superalgebra structures},
\newblock arXiv preprint arXiv:1610.02290.
\bibitem{Mikhalev}
E. A. Vasilieva and A. A. Mikhalev,
\newblock {\it Free left-symmetric superalgebras},
\newblock Fund. and Applied Math
\newblock {\bf  2} (1996), 611--613.

\bibitem{Vinberg}
E. B. Vinberg,
\newblock {\it The theory of homogeneous cones},
\newblock Trudy Moskov
\newblock Mat Obsc {\bf 12} (1963), 303--358.

\bibitem{S-Wang-S-Guo}
S. Wang and S. Guo,
\newblock{\it BiHom-Lie superalgebra structures},
\newblock arXiv preprint
\newblock  arXiv:1610.02290v1 (2016).


%

\end{thebibliography}
\end{document}